\documentclass{article}

\usepackage[scaled=0.92]{helvet}
\usepackage{amssymb,latexsym,amsmath,amsxtra}
\usepackage{amsfonts}
\usepackage{xspace}
\usepackage[final]{graphics}
\usepackage{graphicx,color,psfrag} 
\usepackage{subfig}

\usepackage{amssymb}
\usepackage{latexsym}
\usepackage{amsmath}
\usepackage{xspace}
\usepackage[final]{graphics}
\usepackage{graphicx,color,psfrag} 
\usepackage{subfig}
\usepackage[parfill]{parskip}

\usepackage{amsthm}
\theoremstyle{plain}
\newtheorem{theorem}{Theorem}[section]
\newtheorem{corollary}{Corollary}[section]
\newtheorem{lemma}{Lemma}[section]
\newtheorem{proposition}{Proposition}[section]

\theoremstyle{remark}
\newtheorem{remark}{Remark}[section]
\theoremstyle{definition}

\usepackage{listings}
 \usepackage{color}
 \usepackage{textcomp}
\renewcommand{\vec}[1]{\boldsymbol{#1}}

\newcommand{\paren}[1]{\left(#1\right)}
\newcommand{\brac}[1]{\left[#1\right]}
\newcommand{\E}{\mathbb{E}}
\newcommand{\R}{\mathbb{R}}

\newcommand{\avg}[1]{\E\brac{#1}}

\newcommand{\D}[2]{\frac{d#1}{d#2}}
\newcommand{\DD}[2]{\frac{d^{2}{#1}}{d{#2}^{2}}}
\newcommand{\DO}[1]{\frac{d}{d#1}} 
\newcommand{\PD}[2]{\frac{\partial#1}{\partial#2}}
\newcommand{\PDO}[1]{\frac{\partial}{\partial#1}} 

\newcommand{\lap}[1]{\Delta#1}

\newcommand{\abs}[1]{\left|#1\right|}
\newcommand{\norm}[1]{\Vert#1\Vert}

\DeclareMathOperator{\erfc}{erfc}

\newcommand{\ind}{\mathbf{1}}
\newcommand{\rb}{r_{\textrm{b}}}

\DeclareMathOperator{\prob}{Prob}

\newcommand{\pl}{p_{\lambda}}

\newcommand{\Td}{T_{\textrm{Doi}}}
\newcommand{\Ts}{T_{\textrm{Smol}}}

\newcommand{\vq}{\vec{q}}

\newcommand{\vqa}{\vec{q}^{a}}
\newcommand{\vqb}{\vec{q}^{b}}
\newcommand{\vqc}{\vec{q}^{c}}

\DeclareMathOperator{\card}{card}

\def\R{\mathbb{R}}

\newcommand{\Z}{\mathbb{Z}}

\newcommand{\til}[1]{\widetilde{#1}} 

\DeclareMathOperator{\sech}{sech}

\begin{document}

\title{A Comparison of Bimolecular Reaction Models for Stochastic
  Reaction Diffusion Systems}



\author{I. C. Agbanusi\thanks{Department of Mathematics and
    Statistics, Boston University, 111 Cummington St., Boston, MA
    02215 (agbanusi@math.bu.edu)} \, and S. A. Isaacson\thanks{
    Department of Mathematics and Statistics, Boston University, 111
    Cummington St., Boston, MA 02215 (isaacson@math.bu.edu)}}

\maketitle

\begin{abstract}
  Stochastic reaction-diffusion models have become an important tool
  in studying how both noise in the chemical reaction process and the
  spatial movement of molecules influences the behavior of biological
  systems. There are two primary spatially-continuous models that have
  been used in recent studies; the diffusion limited reaction model of
  Smoluchowski, and a second approach popularized by Doi. Both models
  treat molecules as points undergoing Brownian motion.  The former
  represents chemical reactions between two reactants through the use
  of reactive boundary conditions, with two molecules reacting
  instantly upon reaching a fixed separation (called the
  reaction-radius).  The Doi model uses reaction potentials, whereby
  two molecules react with a fixed probability per unit time,
  $\lambda$, when separated by less than the reaction radius. In this
  work we study the rigorous relationship between the two models.  For
  the special case of a protein diffusing to a fixed DNA binding site,
  we prove that the solution to the Doi model converges to the
  solution of the Smoluchowski model as $\lambda \to \infty$, with a
  rigorous $O(\lambda^{-\frac{1}{2} + \epsilon})$ error bound (for any
  fixed $\epsilon > 0$).  We investigate by numerical simulation, for
  biologically relevant parameter values, the difference between the
  solutions and associated reaction time statistics of the two models.
  As the reaction-radius is decreased, for sufficiently large but
  fixed values of $\lambda$, these differences are found to increase
  like the inverse of the binding radius.
\end{abstract}

\section{Introduction}
Stochastic reaction-diffusion models have become a popular tool for
modeling biological systems in which both noise in the chemical
reaction process and the spatial diffusion of molecules play an
important roll. Such models have been used in a multitude of recent
studies, examining the dynamics of synaptic
transmission~\cite{BartolMCellModel2012}; the MinCDE system in
bacteria~\cite{ElfIEESys04}; how proteins search for DNA binding
sites~\cite{IsaacsonPNAS2011}; and studies of signaling in the
cell membrane~\cite{Dushek2011}.  There are three primary
stochastic reaction-diffusion models that these studies have made use
of; the diffusion limited reaction model of
Smoluchowski~\cite{SmoluchowskiDiffLimRx,KeizerJPhysChem82}, what we
call the Doi
model~\cite{TeramotoDoiModel1967,DoiSecondQuantA,DoiSecondQuantB}, and
the reaction diffusion master equation
(RDME)~\cite{GardinerRXDIFFME,ErbanRDMEReview,OthmerRDME2012}.

In the Doi and Smoluchowski models, the positions of molecules are
represented as points undergoing Brownian motion.  Bimolecular
reactions between two molecules in the Doi model occur with a fixed
probability per unit time when two reactants are separated by
\emph{less} than some specified ``reaction radius''.  The Smoluchowski
model differs by representing bimolecular reactions in one of two
ways; either occurring instantaneously, or with fixed probability per
unit time, when two reactants' separation is \emph{exactly} the
reaction-radius~\cite{SmoluchowskiDiffLimRx,KeizerJPhysChem82}.  In
this work we focus on the former case (often called a pure absorption
reaction). For both models unimolecular reactions represent internal
processes. They are assumed to occur with exponentially distributed
times based on a specified reaction-rate constant. For general
chemical systems, both the Doi and Smoluchowski models can be
described by, possibly infinite, systems of partial integral
differential equations (PIDEs) for the probability densities of
having a given number of each chemical species and the corresponding
locations of each molecule.

The RDME is spatially discrete, and given by a, possibly infinite,
system of ODEs for the numbers of each chemical species located at
each lattice site. It can be interpreted as an extension of the
non-spatial chemical master equation
(CME)~\cite{GardinerRXDIFFME,McQuarrieJAppProb,GardinerHANDBOOKSTOCH,VanKampenSTOCHPROCESSINPHYS}
model for stochastic chemical kinetics.  \emph{Formally,} the RDME has
been shown to be an approximation to both the Doi
model~\cite{IsaacsonRDMENote} and the Smoluchowski
model~\cite{IsaacsonRDMELims,IsaacsonRDMELimsII,ElfPNASRates2010,Hellander:2012jk}
for appropriately chosen lattice spacings. These approximations break
down for systems involving bimolecular reactions,
in two or more dimensions, \emph{in the limit} that the lattice
spacing approaches zero~\cite{IsaacsonRDMELims,Hellander:2012jk}.
(Recently we have suggested a new convergent RDME (CRDME) that
converges to the Doi model as the lattice spacing is taken to
zero~\cite{IsaacsonCRDME12}.)

There are a plethora of numerical methods and simulation packages that
have been developed to study biological systems based on one of the
Smoluchowski, Doi, or RDME models. These include the $\lambda$-$\rho$
method~\cite{ErbanChapman2009, ErbanChapman2011}, the
CRDME~\cite{IsaacsonCRDME12}, the
FPKMC~\cite{DonevJCP2010,WoldeEgfrdPNAS2010}, MCELL~\cite{Bartol2008},
MesoRD~\cite{MesoRDFange2012}, Smoldyn~\cite{AndrewsBrayPhysBio2004},
STEPS~\cite{STEPSHepburn2012dg}, and URDME~\cite{URDME2012}. While
these numerical methods and software packages have been used in many
modeling efforts, it is still an open question how exactly the
underlying mathematical models they approximate are rigorously
related. Moreover, to compare the results of modeling studies it would
be helpful to understand how to choose parameters in the Doi (resp.
Smoluchowski) model to accurately approximate the Smoluchowski (resp.
Doi) model.

To address these questions, we investigate the rigorous relationship
between the Doi and (pure absorption) Smoluchowski models. We begin in
the next section by giving a brief overview of the general Doi and
Smoluchowski models for the bimolecular reaction $\textrm{A} +
\textrm{B} \to \textrm{C}$. (The two models are identical for systems
with only first or zeroth order reactions.) We then consider the
special case of a single protein searching for a DNA binding site
within the nucleus of a eukaryotic cell in Section~\ref{S:radSymEqs}.
The binding site is located at the origin, and the nucleus is modeled
as a concentric sphere of radius $R$.  With the further assumption
that the initial distribution of the protein's position is rotationally
invariant, the Doi and Smoluchowski models can be simplified to
spherically-symmetric diffusion equations. For the Doi model the
diffusing protein is allowed to bind with probability per unit time
$\lambda$ when within a reaction-radius, $\rb$, of the binding site.
This leads to a ``reaction potential'' in the PDE for the Doi model.
In the Smoluchowski model the protein reacts instantly upon reaching a
separation $\rb$ from the binding site. This leads to the replacement
of the reaction potential with a zero Dirichlet boundary condition on
the sphere of radius $\rb$ around the origin.

Denote by $\pl(r,t)$ the spherically symmetric probability density
that the diffusing molecule is $r$ from the origin at time $t$ in the
Doi model, and by $\rho(r,t)$ the corresponding probability density in
the Smoluchowski model.  In Section~\ref{S:numerics} we numerically
calculate the difference between $\pl(r,t)$ and $\rho(r,t)$ for
biologically relevant parameter values.  We find that $\pl(r,t) \to
\rho(r,t)$ uniformly in $r$ and $t$ as $\lambda \to \infty$, with an
empirical convergence rate that is $O(\lambda^{-\frac{1}{2}})$. The
same empirical convergence rate is observed for the corresponding
binding time distributions and mean binding times of the two models.
For sufficiently large, but fixed, values of $\lambda$, as $\rb$ is
decreased the difference between the probability densities, binding
time distributions, and mean binding times increase like $\rb^{-1}$.

These results motivate our studies in Section~\ref{S:convResults},
where we rigorously prove the convergence of $\pl(r,t)$ to $\rho(r,t)$
as $\lambda \to \infty$ with an $O(\lambda^{-\frac{1}{2} + \epsilon})$
error bound (for all $\epsilon > 0$). Let $\mu_n$ denote the $n$th
eigenvalue of the generator of the Doi model
(see~\eqref{eq:doiEigfuncEq}), with $\alpha_n$ the $n$th eigenvalue of
the generator of the Smoluchowski model
(see~\eqref{eq:smolEigfuncEq}). Our approach is to first prove that
for $\lambda$ sufficiently large, if $\alpha_n \leq M(\lambda)$ with
$M(\lambda)$ a specified increasing function of $\lambda$, then
\begin{equation*}
  \abs{\alpha_n - \mu_n} \leq \frac{C}{\lambda^{\frac{1}{2} - \epsilon}}
\end{equation*}
for any fixed $\epsilon > 0$.  The precise statement of this result is
given in Theorem~\ref{thm:EigenConv}. This theorem is then used to
show the corresponding eigenfunction convergence result in
Lemma~\ref{lem:eigen_func_1}.  Denote by $\psi_n(r)$ the $n$th
eigenfunction of the generator of the Doi
model~\eqref{eq:doiEigfuncEq}, and $\phi_n(r)$ the corresponding
eigenfunction of the Smoluchowski model~\eqref{eq:smolEigfuncEq}.  We
prove for $\lambda$ sufficiently large and $\mu_n < M(\lambda)$ that
\begin{equation*}
  \sup_{r \in \brac{\rb,R}} \abs{\phi_n(r) - \psi_n(r)} \leq \frac{C}{\lambda^{\frac{1}{2} - \epsilon}}.
\end{equation*}
Finally, the preceding two results are used to prove that
\begin{equation} \label{eq:unifConvRate}
  \sup_{t \in \paren{\delta,\infty}} \sup_{r \in \paren{\rb,R}} 
  \abs{ \pl(r,t) - \rho(r,t) }
  \leq \frac{C}{\lambda^{\frac{1}{2} - \epsilon}},
\end{equation}
for any $\epsilon > 0$ and $\delta > 0$ fixed.  The precise statement
of this result is given in Theorem~\ref{thm:densityConv}.

It should be noted that the approximation of Dirichlet boundary
conditions by reaction potentials is a well-studied problem in the
context of the large-coupling limit in quantum
physics~\cite{GlimmJaffeBook}.  Our approach of proving convergence
through successive eigenvalue, eigenfunction, and PDE solution
estimates differs from the more standard resolvent and path integral
estimates~\cite{TaylorPDEsV2,Demuth:1980tj,BelhadjLargeCouplingRev11}.
A similar convergence rate of the Doi model solution to the
Smoluchowski model solution was proven in $L^2(\R^3)$
in~\cite{Demuth:1993ug} (as opposed to the uniform convergence
rate~\eqref{eq:unifConvRate} we prove in a spherical domain).  Denote
by $\ind_{\brac{0,\rb}}(r)$ the indicator function of the interval,
$\brac{0,\rb}$. Convergence rates for eigenvalues of the general
one-dimensional operator $-\DD{}{x} + V(x) + \lambda W(x)$ as $\lambda
\to \infty$, for $x \in \R$, were proven
in~\cite{Simon1DLargeCoupling1988ua}. In contrast, in
Section~\ref{S:convResults} we study the spherically symmetric
operator arising in the Doi model, $-\DD{}{r} - \frac{2}{r} \D{}{r} +
\lambda \ind_{\brac{0,\rb}}(r)$ for $r \in \left[ 0, R \right)$ with
a Neumann boundary condition at $R$, directly.

\section{General Doi and Smoluchowski Models}
\label{S:mathModelFormulation} We first illustrate how the bimolecular
reaction $\textrm{A} + \textrm{B} \to \textrm{C}$ would be described
by the Doi and Smoluchowski models in $\R^3$.  In the Doi model, bimolecular
reactions are characterized by two parameters; the separation at which
molecules may begin to react, $\rb$, and the probability per unit time
the molecules react when within this separation, $\lambda$.  In the
Doi and Smoluchowski models, when a molecule of species $\textrm{A}$ and a
molecule of species $\textrm{B}$ react we assume the $\textrm{C}$
molecule they produce is placed midway between them.

We now formulate the Doi model as an infinite coupled system of PIDEs.
Let $\vqa_{l} \in \R^3$ denote the position of the $l$th molecule of
species $\textrm{A}$ when the total number of molecules of species
$\textrm{A}$ is $a$. The state vector of the species $\textrm{A}$
molecules is then given by $\vqa = (\vqa_{1},\dots,\vqa_{a}) \in
\R^{3a}$.  Define
$\vqb$ and $\vqc$ similarly.
We denote by $f^{(a,b,c)}(\vqa,\vqb,\vqc,t)$ the probability density
for there to be $a$ molecules of species $\textrm{A}$, $b$ molecules
of species $\textrm{B}$, and $c$ molecules of species $\textrm{C}$ at
time $t$ located at the positions $\vqa$, $\vqb$, and $\vqc$.
Molecules of the same species are assumed indistinguishable.  In the
Doi model the evolution of $f^{(a,b,c)}$ is given by
\begin{equation} \label{eq:doiFormEvolEq}
    \PD{f^{(a,b,c)}}{t}\paren{\vqa,\vqb,\vqc,t} = \paren{L + R} f^{(a,b,c)}\paren{\vqa,\vqb,\vqc,t}.
\end{equation}
Note, with the subsequent definitions of the operators $L$ and $R$
this will give a coupled system of PIDEs over all possible values of
$(a,b,c)$. 
More general systems that allow unbounded production of certain
species would result in an infinite number of coupled PIDEs.  The
diffusion operator, $L$, is defined by
\begin{equation} \label{eq:doiFormLapOp}
  \begin{aligned}
  L f^{(a,b,c)}\paren{\vqa,\vqb,\vqc,t} &=
  \paren{D^{\textrm{A}} \sum_{l=1}^{a} \Delta_{l}^{a}
  + D^{\textrm{B}} \sum_{m=1}^{b} \Delta_{m}^{b} + D^{\textrm{C}} \sum_{n=1}^{c} \Delta_{n}^{c}}  f^{(a,b,c)}\paren{\vqa,\vqb,\vqc,t},
\end{aligned}
\end{equation}
where $\Delta_{l}^{a}$ denotes the Laplacian in the coordinate
$\vqa_{l}$ and $D^{\textrm{A}}$ the diffusion constant of species
$\textrm{A}$. $D^{\textrm{B}}$, $D^{\textrm{C}}$, $\lap_{m}^{b}$, and
$\lap_{n}^{c}$ are defined similarly.

To define the reaction operator, $R$, we must introduce notations to
represent removing or adding a specific molecule to the state $\vqa$.
Let
\begin{align*}
  \vqa \setminus \vqa_{l} = \paren{\vqa_{1},\dots,\vqa_{l-1},\vqa_{l+1},\dots,\vqa_{a}}, &&   \vqa \cup \vq = \paren{\vqa_{1},\dots,\vqa_{a},\vq}.
\end{align*}
$\vqa \setminus \vq$ will denote $\vqa$ with any one component with
the value $\vq$ removed. 
Denote by $\ind_{\brac{r \leq \rb}}(r)$ the indicator function of the
interval $\brac{0,\rb}$, and let $B_{l}^{c} = \{ \vq \in \R^{3} |
\abs{\vq-\vqc_{l}} \leq \rb/2\}$ label the set of points a reactant
could be at to produce a molecule of species $\textrm{C}$ at $\vqc_l$.
The Doi reaction operator, $R$, is then
\begin{multline} \label{eq:doiRxOp}
    (R f^{(a,b,c)})\paren{\vqa,\vqb,\vqc,t} = 
    \lambda \Bigg[ \sum_{l=1}^{c}  
    \int_{\vq \in B_{l}^{C}} f^{(a+1,b+1,c-1)} \paren{\vqa \cup \vq, \vqb \cup \paren{2 \vqc_{l} - \vq}, \vqc \setminus \vqc_{l}, t} \, dB_{l}^{c} \\
    - \sum_{l=1}^a \sum_{l'=1}^b 
    \ind_{\brac{0,\rb}} \big( \big| \vqa_l-\vqb_{l'} \big| \big) f^{(a,b,c)}(\vqa,\vqb,\vqc,t) \Bigg].
\end{multline}
In the Fock space representation
of~\cite{DoiSecondQuantA,DoiSecondQuantB}, the term $R f^{(a,b,c)}$ is
called a reaction potential~\cite{DoiSecondQuantB}. 

The Smoluchowski model would modify~\eqref{eq:doiFormEvolEq} by adding
a reactive boundary condition and changing the reaction
operator~\eqref{eq:doiRxOp}. The reactive boundary condition, often
called a ``pure absorption'' boundary condition, is the Dirichlet
boundary condition that
\begin{equation} \label{eq:scdlrRxBC}
  f^{(a,b,c)}\paren{\vqa,\vqb,\vqc,t} = 0, \quad \text{if } \abs{\vqa_{l}-\vqb_{m}} = \rb, \text{ for any } l,m.
\end{equation}
We define by $S_{l}^{c} = \{ \vq \in \R^{3} | \abs{\vq-\vqc_{l}} = \rb
/2\}$ the set of points reactants could be at to produce a molecule of
species $\textrm{C}$ at $\vqc_{l}$. For $\vq \in S_{l}^{c}$, denote by
$J^{(a+1,b+1,c-1)}_{l}\paren{\vqa \cup \vq, \vqb \cup \paren{2
    \vqc_{l} - \vq}, \vqc \setminus \vqc_{l}, t}$ the diffusive flux
along the inward normal to the surface $S_{l}^{c}$ from
$f^{(a+1,b+1,c-1)}\paren{\vqa \cup \vq, \vqb \cup \paren{2 \vqc_{l} -
    \vq}, \vqc \setminus \vqc_{l}, t}$, immediately prior to a
reaction producing $\vqc_{l}$.  The reaction-operator, $R$, is then
\begin{equation} \label{eq:scdlrRxOp}
  (R f^{(a,b,c)})\paren{\vqa,\vqb,\vqc,t} =  \sum_{l=1}^{c}  
  \int_{\vq \in S_{l}^{C}} J^{(a+1,b+1,c-1)}_{l} \paren{\vqa \cup \vq, \vqb \cup \paren{2 \vqc_{l} - \vq}, \vqc \setminus \vqc_{l}, t} \, dS_{l}^{c}.
\end{equation}

Note, in the Smoluchowski model if the molecules are instead assumed to react
with some fixed probability per unit time upon collision, a ``partial
absorption'' Robin boundary condition would be used,
see~\cite{KeizerJPhysChem82, ElfPNASRates2010}.  We do not consider
this more general reaction mechanism in this work.

\section{Diffusion of a Protein to a Fixed DNA Binding Site} \label{S:radSymEqs}
To study the rigorous relationship between the Doi and Smoluchowski
models we investigate the special case of the chemical reaction
$\textrm{A} + \textrm{B} \to \varnothing$, with only one molecule of
species $\textrm{A}$ and one molecule of species $\textrm{B}$. We
further assume the molecule of species $\textrm{B}$ is located at the
origin and stationary ($D^{\textrm{B}} = 0$). The molecule of species
\textrm{A} is assumed to move within a sphere of radius $R$ centered
on the \textrm{B} molecule. We denote the diffusion constant of the
\textrm{A} molecule by $D$ and the reaction radius by $\rb$.  While
idealized, this special case can be interpreted as a model for the
diffusion of a DNA binding protein to a fixed DNA binding site
(located at the center of a nucleus).

We now formulate the Doi and Smoluchowski models in this special case,
with the additional assumption of spherical symmetry. This assumption
will hold whenever the initial distribution of the \textrm{A} molecule
is spherically-symmetric about the origin. Denote by $\rho(r,t)$ the
spherically-symmetric probability density that the \textrm{A} molecule
is a distance $r$ from the origin and has not reacted with the
\textrm{B} molecule at time $t$. Then the Smoluchowski
model~\eqref{eq:doiFormEvolEq}, \eqref{eq:doiFormLapOp},
\eqref{eq:scdlrRxBC}, \eqref{eq:scdlrRxOp} reduces to
\begin{align} \label{eq:radpureSmol}
  \PD{\rho}{t} &= D \Delta_{r} \rho, \quad \rb < r < R, \, t > 0,
\end{align}
where $\Delta_r$ denotes the spherically symmetric Laplacian in
three-dimensions,
\begin{equation*}
\Delta_r \equiv \frac{1}{r^2}\PDO{r}\paren{r^2\PDO{r}}.
\end{equation*}
This equation is coupled with the reactive Dirichlet boundary
condition, and zero Neumann boundary condition (so that the molecule
remains in the ``nucleus'')
\begin{align*}
  \rho(\rb,t) &= 0, && \PD{\rho}{r}(R,t) = 0, \quad t > 0.
\end{align*}
Finally, we denote by $g(r)$ the initial condition, $\rho(r,0) =
g(r)$, with the normalization $\int_{\rb}^{R}g(r) r^2 \, dr = 1$.

Let $\pl(r,t)$ label the corresponding spherically-symmetric probability
density for the Doi model. In the special case we are considering, the
PIDEs for the Doi model~\eqref{eq:doiFormEvolEq},
\eqref{eq:doiFormLapOp}, \eqref{eq:doiRxOp} reduce to
\begin{align} \label{eq:radpureDoi}
  \PD{\pl}{t} &= D \Delta_{r} \pl - \lambda \, \ind_{\brac{0,\rb}}(r) \, \pl(r,t), \quad 0 \leq r < R, \, t > 0,
\end{align}
with the Neumann boundary condition,
\begin{align*}
  \PD{\pl}{r}(R,t) = 0, \quad t > 0,
\end{align*}
and the initial condition that 
\begin{align*}
  \pl(r,0) = \tilde{g}(r) = 
  \begin{cases}
    g(r),  &r > \rb,\\
    0, &r \leq \rb.
  \end{cases}
\end{align*}

For simplicity, in what follows, we assume $D=1$.
Equations~\eqref{eq:radpureSmol} and~\eqref{eq:radpureDoi} can be
solved explicitly by separating variables. In
solving~\eqref{eq:radpureDoi} we impose continuity of the function and
its derivative across the surface of discontinuity as justified by the
results of~\cite{Girsanov:1960dc} and~\cite{Olenik:1959dc}. The
computations are standard so we give only the final results.  Denote
by $(u(r),v(r))=\int_{0}^{R} {u(r)v(r)r^2} \,dr$ the usual $L^2$ inner
product. We can then write the solutions as
\begin{equation} \label{eq:smolEigFuncExpan}
\rho(r,t) = \sum_{n=1}^{\infty}{a_n(\tilde{g}(r),\phi_{n}(r))\phi_{n}(r)e^{-\alpha_{n}t}}
\end{equation}
and 
\begin{equation} \label{eq:doiEigFuncExpan}
  \pl(r,t) = \sum_{n=1}^{\infty}{b_n(\tilde{g}(r),\psi_{n}(r))\psi_{n}(r)e^{-\mu_{n}(\lambda)t}}.
\end{equation}
Here
\begin{equation*}
  \phi_n(r) =\frac{1}{r}\brac{\frac{\sin(\sqrt{\alpha_n}(R-r))}{R\sqrt{\alpha_n}}-\cos(\sqrt{\alpha_n}(R-r))},\quad \rb<r<R
\end{equation*}
are the eigenfunctions for the Smoluchowski model
\eqref{eq:radpureSmol}, satisfying
\begin{align} \label{eq:smolEigfuncEq}
  -\Delta_r \phi_n(r) &= \alpha_n \phi_n(r), \quad \rb < r < R,
\end{align}
with the boundary conditions $\phi_n(\rb) = 0$ and $\PD{\phi_n}{r}(R)
= 0$.  We extend these functions to $\brac{0,R}$ by defining
$\phi_n(r) = 0$ for $r \in \brac{0,\rb}$.  The corresponding
eigenvalues, $\alpha_{n}$, solve the equation $f(\alpha_n)=0$, where
\begin{equation} 
  f(\mu) = \frac{R\sqrt{\mu}-\tan(\sqrt{\mu}(R-\rb))}{R\mu \tan(\sqrt{\mu}(R-\rb))+\sqrt{\mu}}.
\end{equation}

The eigenfunctions for the Doi model~\eqref{eq:radpureDoi} are given by
\begin{equation*}
  \psi_n(r) =
    \begin{cases}
    \psi^{in}_{n}(r), & 0<r<\rb,\\
    \psi^{out}_{n}(r), & \rb \leq r<R,
    \end{cases}
\end{equation*}
where
\begin{align}
  \psi^{in}_{n}(r) =
    \begin{cases}
     \displaystyle\frac{\tfrac{1}{R\sqrt{\mu_n}}\sin(\sqrt{\mu_n}(R-\rb))-\cos(\sqrt{\mu_n}(R-\rb))}{\sinh(\rb\sqrt{\lambda-\mu_n})}\paren{\frac{\sinh(r\sqrt{\lambda-\mu_n})}{r}}, &\mu_n<\lambda,\\
     \displaystyle\frac{\tfrac{1}{R\sqrt{\mu_n}}\sin(\sqrt{\mu_n}(R-\rb))-\cos(\sqrt{\mu_n}(R-\rb))}{\sin(\rb\sqrt{\mu_n-\lambda})}\paren{\frac{\sin(r\sqrt{\mu_n-\lambda})}{r}}, &\mu_n>\lambda,
     \end{cases}
\end{align}
and
\begin{equation}
\psi^{out}_{n}(r) =\frac{1}{r}\brac{\frac{\sin(\sqrt{\mu_n}(R-r))}{R\sqrt{\mu_n}}-\cos(\sqrt{\mu_n}(R-r))}.
\end{equation}
They satisfy the equation
\begin{align} \label{eq:doiEigfuncEq}
  -\Delta_r \psi_n(r) + \lambda \ind_{\brac{0,\rb}} \psi_n(r) &= \mu_n \psi_n(r), \quad 0 \leq r < R,
\end{align}
with the Neumann boundary condition $\PD{\psi_n}{r}(R) = 0$.

The Doi eigenvalues, $\mu_{n}$, solve the equation $f(\mu_n) = A(\mu_n,\lambda)$ with 
\begin{align}
  A(\mu,\lambda)= 
  \begin{cases}
   \displaystyle \sqrt{\frac{1}{\lambda-\mu}}\tanh\paren{\sqrt{\lambda-\mu}\rb}, & \mu<\lambda,\\
   \displaystyle \sqrt{\frac{1}{\mu - \lambda}}\tan\paren{\sqrt{\mu-\lambda}\rb}, & \mu>\lambda.
  \end{cases}
\end{align}
We denote the dependence of the eigenvalues on $\lambda$ by
$\mu_n(\lambda)$.  The constants $a_n$ and $b_n$ are given
by
\begin{align*}
 a_n &=\frac{1}{(\phi_n(r),\phi_n(r))}, &
 b_n &=\frac{1}{(\psi_n(r),\psi_n(r))}.
\end{align*}

\section{Difference Between Smoluchowski and Doi Models for
  Biologically Relevant Parameters} \label{S:numerics} 
If we interpret~\eqref{eq:radpureSmol} and~\eqref{eq:radpureDoi} as
models for the diffusion of a protein that has just entered the
nucleus ($r_0 = R$) to a DNA binding site, then typically $\rb$ would
be between $.1$ and $10 \,
\textrm{nm}$~\cite{KuhnerLexADNABond,Dushek2011,AndrewsBrayPhysBio2004},
$R$ between $1$ and $10 \, \mu m$, and $D$ between $1$ and $20 \, \mu
\textrm{m}^2 \textrm{s}^{-1}$.  In the following we assume that all
spatial units are in micrometers and time is in seconds, with $R = r_0
= 1 \, \mu \textrm{m}$, $D = 10 \, \mu \textrm{m}^2 \textrm{s}^{-1}$,
and $\lambda$ having units of $\textrm{s}^{-1}$. We also assume that
$\pl(r,0) = \rho(r,0) = \delta(r - r_0) / 4 \pi r^2$.  $\pl$ and
$\rho$ are then the same as in the previous section, but rescaled by
$(4 \pi)^{-1}$.

\begin{figure}[t]
  \centering
  \subfloat[]{
    \label{fig:diffEqDensityPlot}
    \scalebox{.55}{\includegraphics{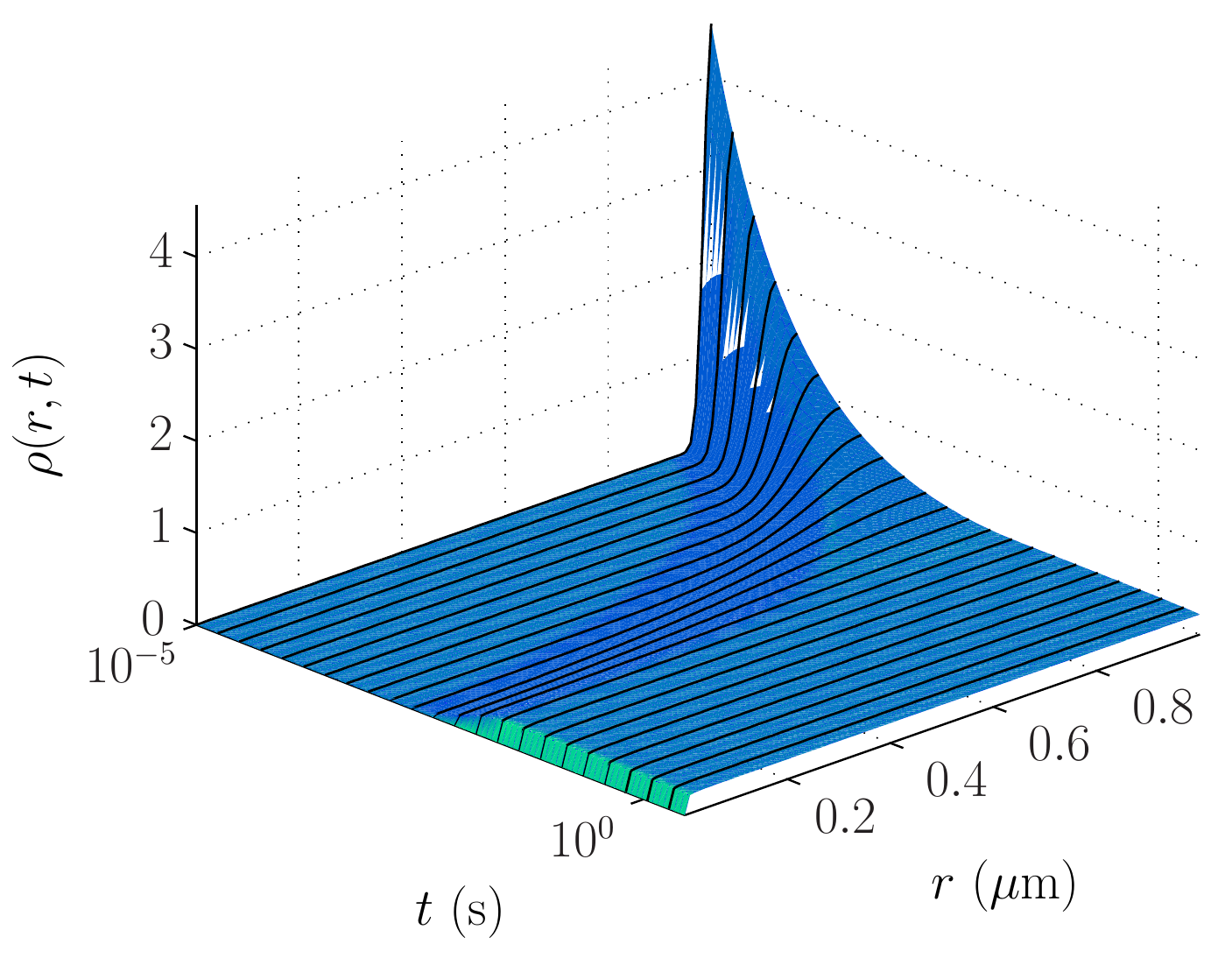}}
  }
  \subfloat[]{
    \label{fig:diffEqDensityErr}
    \scalebox{.55}{\includegraphics{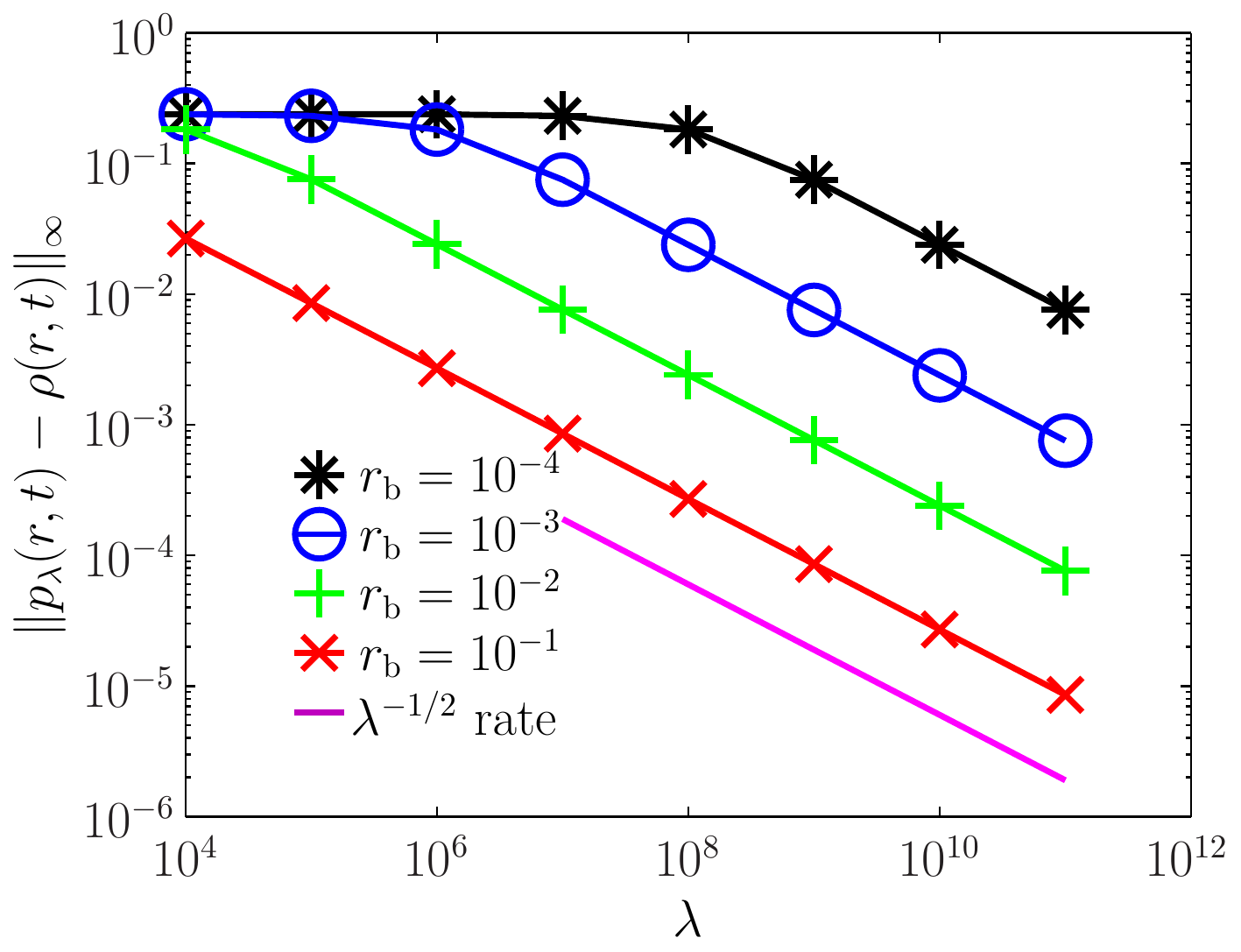}}
  }
  \caption{(a) Solution to the Smoluchowski
    model~\eqref{eq:radpureSmol} for $\rb = 10^{-3} \mu \textrm{m}$.
    Note that the $t$-axis uses a logarithmic scale. (b) Absolute difference
    in $\pl(r,t)$ and $\rho(r,t)$ as $\lambda$ and $\rb$ are varied. }
\end{figure}
We numerically evaluated $\pl(r,t)$ and $\rho(r,t)$ in MATLAB using
the eigenfunction expansions~\eqref{eq:doiEigFuncExpan}
and~\eqref{eq:smolEigFuncExpan}. The series were truncated at the
first term with magnitude smaller than $10^{-10}$. In evaluating these
series numerically it is necessary to calculate a number of the
eigenvalues $\mu_n$ and $\alpha_n$.  For each term of the
eigenfunction expansions the transcendental equations for the
corresponding Doi eigenvalue, $f(\mu_n) = A(\mu_n,\lambda)$, and the
Smoluchowski eigenvalue, $f(\alpha_n) = 0$, were solved to 25 digits
of precision using the Mathematica \texttt{Reduce} function.  In
Figure~\ref{fig:diffEqDensityPlot} we show the solution to the
Smoluchowski model~\eqref{eq:radpureSmol}, $\rho(r,t)$, when $\rb =
10^{-3} \mu \textrm{m}$. For short times the solution is localized
near $R$, while a boundary layer develops near $r = \rb$ as $t$
increases.  Figure~\ref{fig:diffEqDensityErr} shows the maximum
absolute difference between $\pl$ and $\rho$ for a discrete set of
points,
\begin{equation*}
\norm{\pl(r,t) - \rho(r,t)}_{\infty} \equiv \sup_{i} \sup_{j} \abs{\pl(r_i,t_j) - \rho(r_i,t_j)},
\end{equation*}
where $r_i$ and $t_j$ are given in Appendix~\ref{S:appendixB} by
Listings~\ref{lst:rpoints} and~\ref{lst:tpoints}.  For each fixed
value of $\rb$ we see that as $\lambda \to \infty$ the difference
between $\pl$ and $\rho$ converges to zero like $\lambda^{-1/2}$. As
$\rb$ is decreased, the absolute difference increases by approximately
an order of magnitude for each fixed value of $\lambda$ (for $\lambda$
sufficiently large).

\begin{figure}[t]
  \centering
  \subfloat[]{
    \label{fig:bindDistribPlot}
    \scalebox{.53}{\includegraphics{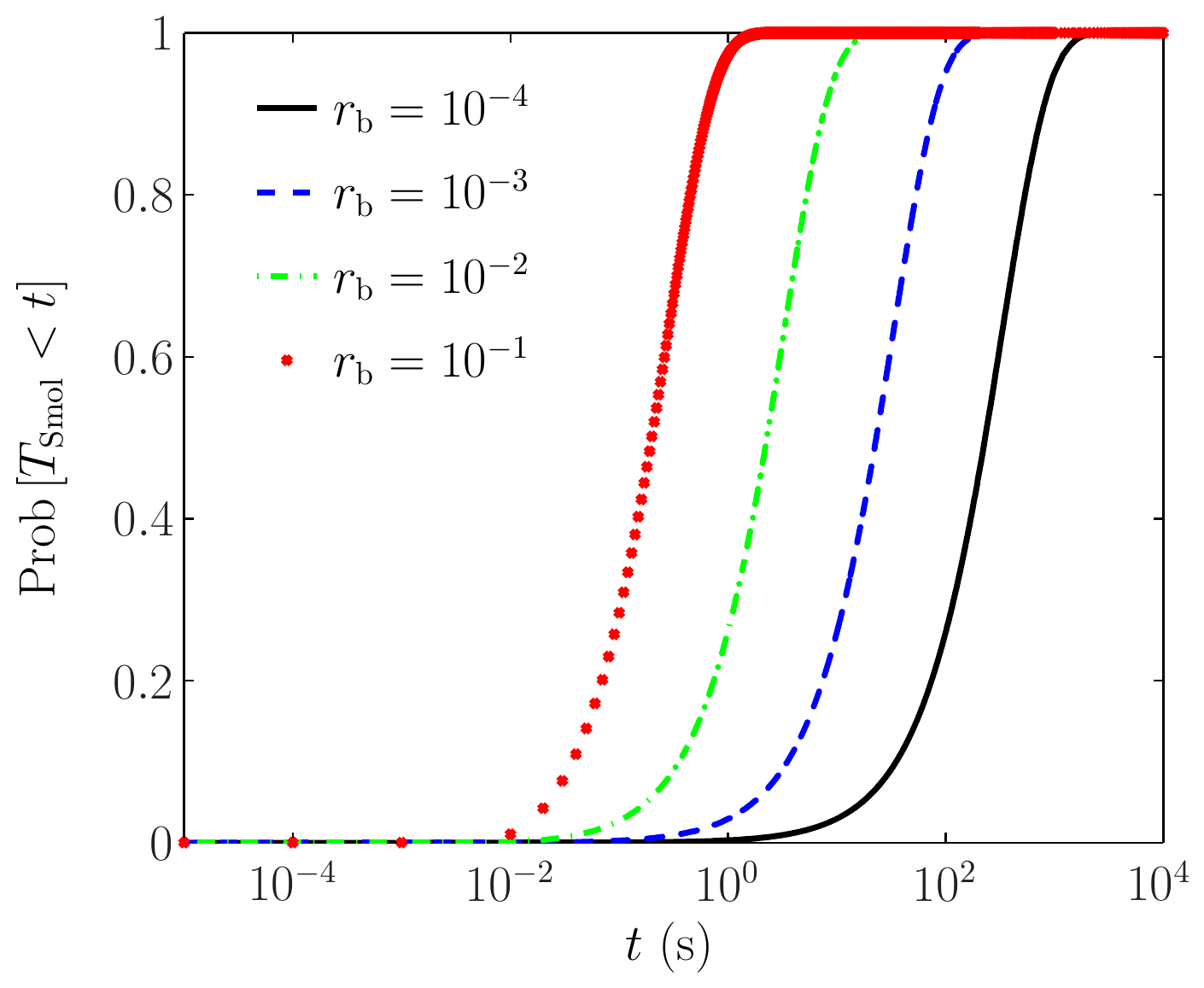}}
  }  
  \subfloat[]{
    \label{fig:bindDistribErr}
    \scalebox{.55}{\includegraphics{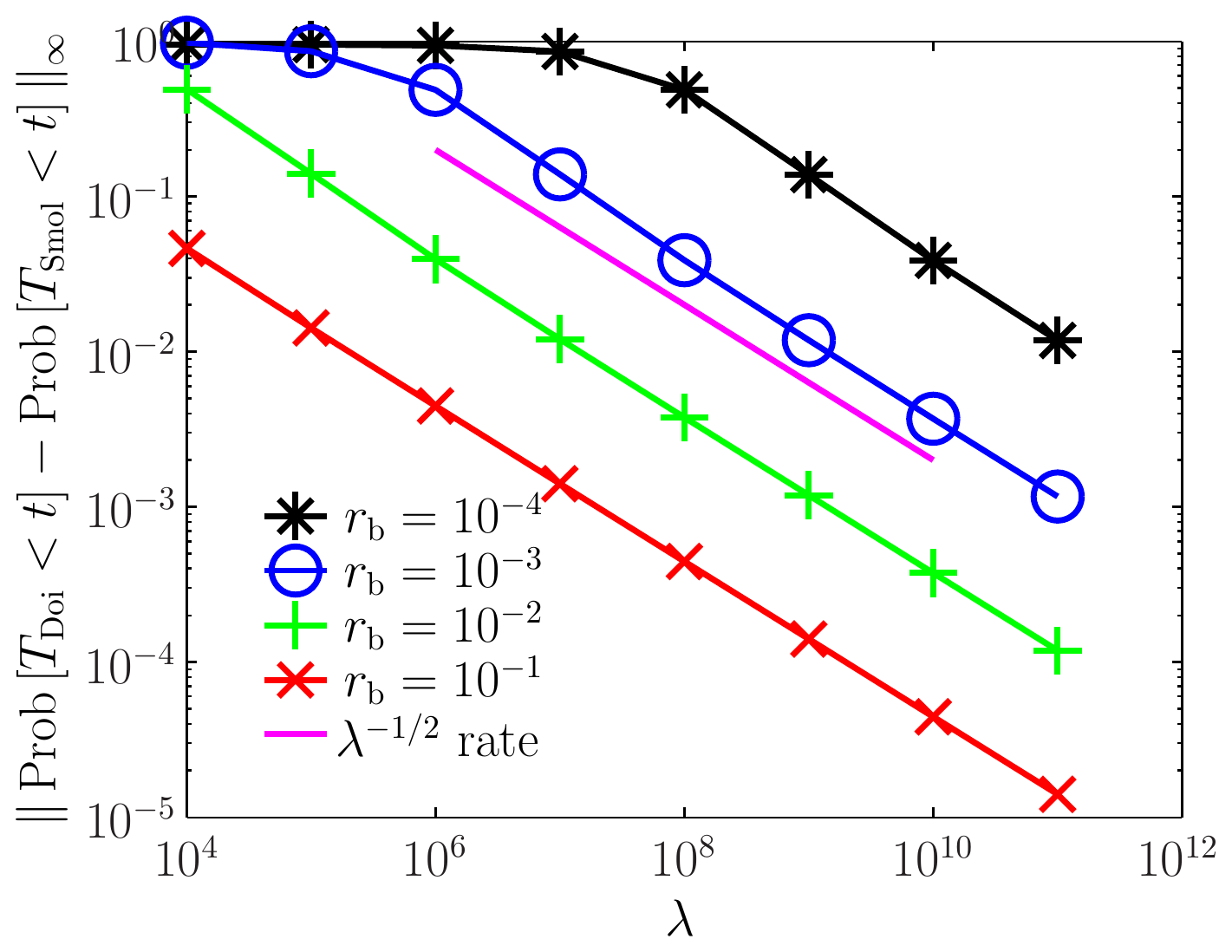}}
  }  
  \caption{(a) Smoluchowski binding time distributions for varying
    $\rb$. Note the $t$-axis is logarithmic. (b) Absolute difference
    in binding time distributions from the Smoluchowski and Doi models
    as $\lambda$ is varied. }
\end{figure}
For many biological models the statistics of the random variable, for
the time at which the diffusing molecule first binds to the binding
site, are of interest. For example, in~\cite{IsaacsonPNAS2011} we
studied how this time was influenced by volume exclusion due to the
spatially varying density of chromatin inside the nucleus of mammalian
cells. We subsequently denote these random variables by $\Td$ and
$\Ts$.  Statistics of these random variables can be calculated from
the cumulative distribution functions
\begin{equation*}
  \prob \brac{\Td < t} = 1 - 4 \pi \int_0^R \pl(r,t) r^2 \, dr,
\end{equation*}
and 
\begin{equation*}
  \prob \brac{\Ts < t} = 1 - 4 \pi \int_{\rb}^R \rho(r,t) r^2 \, dr.
\end{equation*}
We evaluated $\prob \brac{\Td < t}$ and $\prob \brac{\Ts < t}$ by
analytically integrating the eigenfunction
expansions~\eqref{eq:doiEigFuncExpan} and~\eqref{eq:smolEigFuncExpan}
and evaluating the truncated series in MATLAB. (Using the same method
as described above for evaluating $\pl(r,t)$ and $\rho(r,t)$.)

Figure~\ref{fig:bindDistribPlot} shows $\prob \brac{\Ts < t}$ for
varying $\rb$ and demonstrates a constant increase in the binding time
as $\rb$ is decreased (on a logarithmic scale).
Figure~\ref{fig:bindDistribErr} shows the absolute difference in
binding time distributions,
\begin{equation*}
  \norm{ \prob \brac{ \Td < t } - \prob \brac{\Ts < t}}_{\infty} 
  \equiv \sup_{t_j} \abs{ \prob \brac{ \Td < t_j } - \prob \brac{\Ts < t_j} },
\end{equation*}
where $t_j$ is given by Listing~\ref{lst:tpoints} in
Appendix~\ref{S:appendixB}. We again observe an empirical
$\lambda^{-\frac{1}{2}}$ convergence rate of $\prob \brac{\Td < t}$ to
$\prob \brac{\Ts < t}$ as $\lambda \to \infty$.  For a biologically
relevant binding radius of $10^{-3} \mu \textrm{m}$, when $\lambda =
10^{11} \textrm{s}^{-1}$ the absolute difference between the two
distributions is on the order of $10^{-3}$.

\begin{figure}[t]
  \centering
  \subfloat[]{
    \label{fig:meanBindTimePlot}
    \scalebox{.5}{\includegraphics{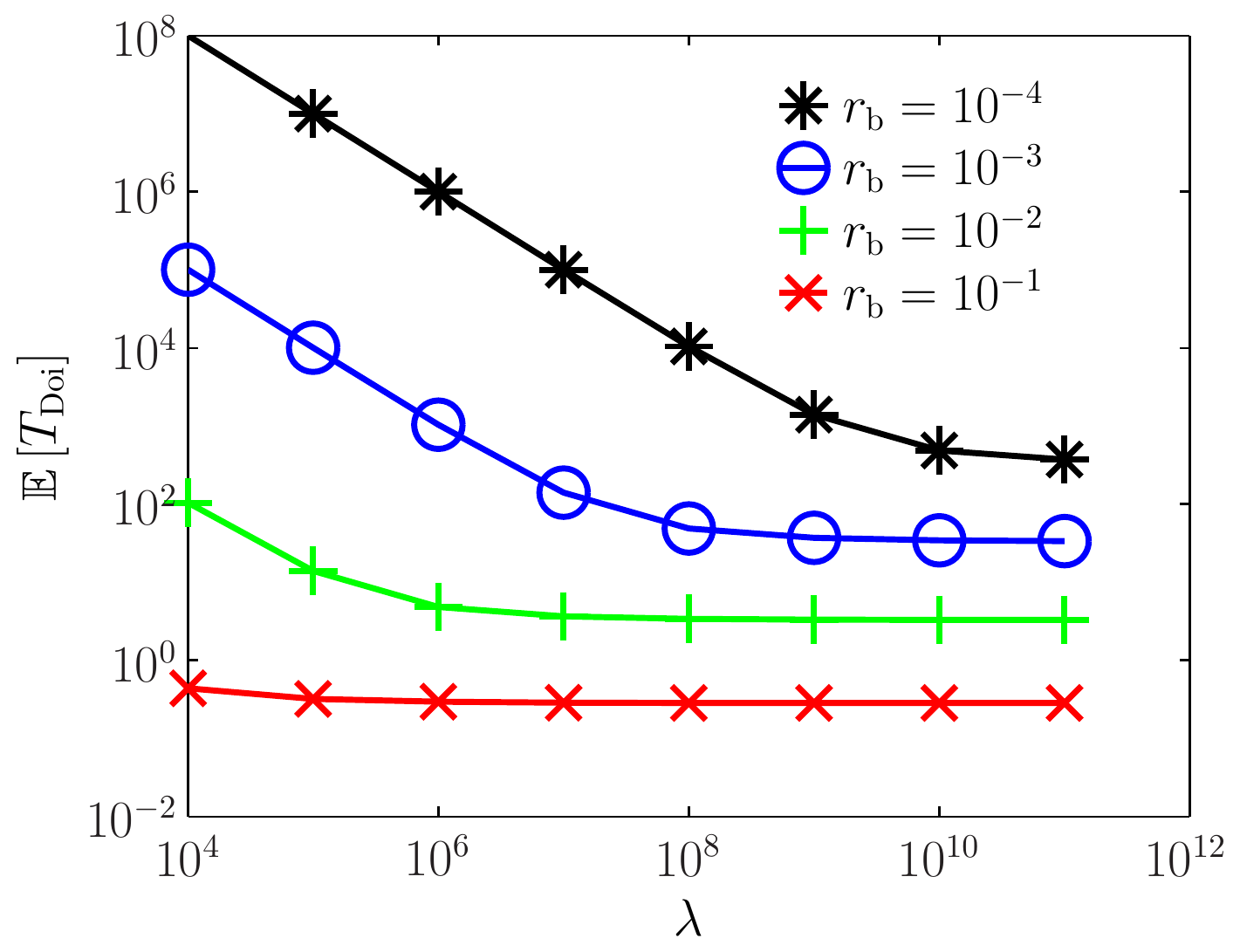}}
  }
  \subfloat[]{
    \label{fig:meanBindTimeErr}
    \scalebox{.5}{\includegraphics{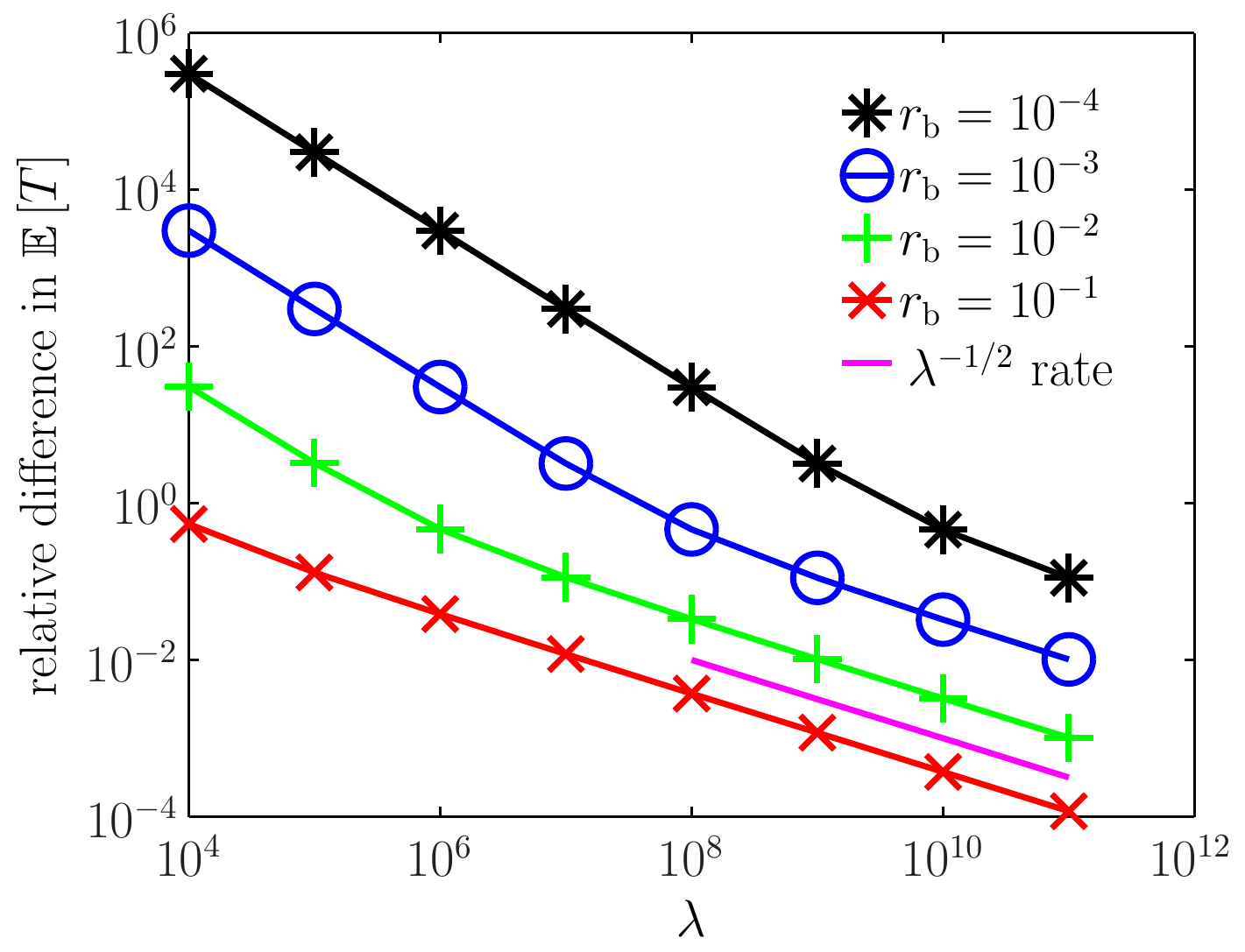}}
  }
  \caption{(a) Mean binding time of Doi model vs $\lambda$. (b)
    Relative difference in mean binding
    times~\eqref{eq:meanBindTimeRelErr}.}
\end{figure}
In addition to the probability densities and binding time
distributions, we also examined the mean time for the diffusing
protein to find the binding site (when starting a distance $r_0$ from
the origin).  The mean binding time for the Doi model can be found in
a simple closed form by solving the corresponding mean first passage
time problem~\eqref{eq:radpureDoiAvg}, see Appendix~\ref{S:appendix},
and is given by
\begin{equation} \label{eq:doiMeanBindTimeSolut}
  \avg{\Td}(r_0) = \int_0^{\infty} \prob \brac{\Td < t} \, dt = 
  \begin{cases}
    u^{-}(r_0), \quad r_0 < \rb, \\
    u^{+}(r_0), \quad r_0 > \rb.
  \end{cases}
\end{equation}
Where, for $\hat{\lambda} = \lambda / D$, we have that
\begin{align*}
 u^{-}(r_0) &=\frac{1}{D\hat{\lambda}}+\paren{\frac{\sinh(\sqrt{\hat{\lambda}}r_0)}{r_0}}\brac{\frac{R^3-\rb^3}{3D\paren{\rb\sqrt{\hat{\lambda}}\cosh(\sqrt{\hat{\lambda}}\rb) -\sinh(\sqrt{\hat{\lambda}}\rb)}}},\\
 u^{+}(r_0) &=\frac{\rb^2-r_0^2}{6D}+ \frac{R^3}{3D}\brac{\frac{1}{\rb}-\frac{1}{r_0}} + \frac{1}{D\hat{\lambda}}+\paren{\frac{\sinh(\sqrt{\hat{\lambda}}\rb)}{\rb}}\brac{\frac{R^3-\rb^3}{3D\paren{\rb\sqrt{\hat{\lambda}}\cosh(\sqrt{\hat{\lambda}}\rb) -\sinh(\sqrt{\hat{\lambda}}\rb)}}}. 
\end{align*}
Similarly, the mean binding time in the Smoluchowski model can be found
by solving the mean first passage time
problem~\eqref{eq:radpureSmolAvg}, and is given by
\begin{equation} \label{eq:smolMeanBindTimeSolut}
\avg{\Ts}(r_0) = \frac{\rb^2-r_0^2}{6D}+ \frac{R^3}{3D}\brac{\frac{1}{\rb}-\frac{1}{r_0}}, \quad r_0 > \rb.
\end{equation}
Figure~\ref{fig:meanBindTimePlot} shows the mean binding time in the
Doi model as $\lambda$ is varied for $r_0 = R$.

The difference between the two mean binding times, for $r_0 > \rb$, is
then given by
\begin{align} \label{eq:meanBindTimeDifference}
  \abs{\avg{\Td}(r_0) - \avg{\Ts}(r_0)} &= \abs{ \frac{1}{D\hat{\lambda}}+\paren{\frac{\sinh(\sqrt{\hat{\lambda}}\rb)}{\rb}}
  \brac{\frac{R^3-\rb^3}{3D\paren{\rb\sqrt{\hat{\lambda}}\cosh(\sqrt{\hat{\lambda}}\rb) -\sinh(\sqrt{\hat{\lambda}}\rb)}}}}. 
\end{align}
This difference is $O(\lambda^{-\frac{1}{2}})$, consistent with the
results we prove in the next section on the uniform convergence of
$\pl$ to $\rho$ as $\lambda \to \infty$.
\eqref{eq:meanBindTimeDifference} also demonstrates that for large,
but fixed values of $\lambda$ the absolute difference in mean binding
times will increase like $\rb^{-1}$ as $\rb$ is decreased.  In
Figure~\ref{fig:meanBindTimeErr} the relative difference,
\begin{equation} \label{eq:meanBindTimeRelErr}
  \frac{\abs{\avg{\Ts} - \avg{\Td}}}{\avg{\Ts}},
\end{equation}
is graphed as $\lambda$ is varied for $r_0 = R$, illustrating the
$\lambda^{-1/2}$ convergence of $\avg{\Td}$ to $\avg{\Ts}$. For $\rb =
10^{-3} \mu \textrm{m}$ the mean binding times differ by less than $1
\%$ when $\lambda = 10^{11} \, \textrm{s}^{-1}$.

Each of Figures~\ref{fig:diffEqDensityErr},~\ref{fig:bindDistribErr},
and~\ref{fig:meanBindTimeErr} illustrate a $\lambda^{-1/2}$ empirical
convergence rate, consistent with the bound~\eqref{eq:unifConvRate} we
prove in the next section (Theorem~\ref{thm:densityConv}). Moreover,
they demonstrate that as $\rb$ is decreased larger values of $\lambda$
are required to ensure the absolute difference between the two models
remains below a fixed tolerance.  In all three cases, once $\lambda$
is sufficiently large that the $\lambda^{-1/2}$ convergence rate can
be observed, the absolute difference appears to scale like $\rb^{-1}$ as
$\rb$ is decreased.

\section{Rigorous Convergence Results} \label{S:convResults} In this
section we study the rigorous relationship between the Doi
model~\eqref{eq:radpureDoi} and Smoluchowski
model~\eqref{eq:radpureSmol}. In Subsection~\ref{sec:eigen_estimates}
we derive a rigorous error bound on the rate of convergence of the Doi
eigenvalues, $\mu_n(\lambda)$, to the Smoluchowski eigenvalues,
$\alpha_n$, as $\lambda \to \infty$.  In
Subsection~\ref{sec:eigenfunc_estim} we obtain similar estimates for
the convergence of the Doi eigenfunctions, $\psi^{out}_n$, to the
Smoluchowski eigenfunctions, $\phi_n$.  We obtain our main result in
subsection~\ref{sec:main_result} where we use these eigenvalue and
eigenfunction estimates to show the uniform convergence in space and
time of the solution to the Doi model~\eqref{eq:radpureDoi},
$\pl(r,t)$, to the solution of the Smoluchowski
model~\eqref{eq:radpureSmol}, $\rho(r,t)$, as $\lambda \to \infty$.
(With the error bound~\eqref{eq:unifConvRate}.)

\subsection{Eigenvalue Estimates}\label{sec:eigen_estimates}
In this subsection we derive estimates for the difference between the
Doi, $\mu_n(\lambda)$, and Smoluchowski, $\alpha_n$, eigenvalues. We
start by proving some properties of the functions $A(\mu,\lambda)$ and
$f(\mu)$ that we will find useful:
\begin{proposition}
  $A(\mu,\lambda)$ is monotone increasing. Furthermore for $\mu \leq
  \lambda$, $A$ is positive.
\end{proposition}
\begin{proof}
  Positivity is trivial since $\tanh$ is positive.  Note also that
  $A(0,\lambda)
  =\tanh(\sqrt{\lambda}\rb) / \sqrt{\lambda}$. A simple
  computation shows that for $\mu<\lambda$
\begin{align*}
\DO{\mu}A(\mu;\lambda) = \frac{\tanh\paren{\sqrt{\lambda-\mu}\rb}-\rb\sqrt{\lambda-\mu}\paren{\sech^{2}\paren{\sqrt{\lambda-\mu}\rb}}}{2\paren{\lambda-\mu}^{3/2}}
=\frac{\sinh\paren{2\sqrt{\lambda-\mu}\rb}-2\rb\sqrt{\lambda-\mu}}{4\cosh^{2}\paren{\sqrt{\lambda-\mu}\rb}\paren{\lambda-\mu}^{3/2}}
\end{align*}
and for $\mu>\lambda$ 
\begin{align*}
  \DO{\mu}A(\mu;\lambda) = \frac{\rb\sqrt{\mu-\lambda}\paren{\sec^{2}\paren{\sqrt{\mu-\lambda}\rb}}-\tan\paren{\sqrt{\mu-\lambda}\rb}}{2\paren{\lambda-\mu}^{3/2}}
=\frac{2\rb\sqrt{\mu-\lambda}-\sin\paren{2\sqrt{\mu-\lambda}\rb}}{4\cos^{2}\paren{\sqrt{\mu-\lambda}\rb}\paren{\mu-\lambda}^{3/2}}.
\end{align*}

The result follows since for $u\geq 0$, $\sinh(u)\geq u$ and $\sin(u)\leq u$.
\end{proof}

Let the vertical asymptotes of $f(\mu)$ be denoted by $\beta_n$. They satisfy the equation
\begin{equation*}
R\beta_n \tan(\sqrt{\beta_n}(R-\rb))+\sqrt{\beta_n}=0;\quad \beta_n >0.
\end{equation*}

\begin{proposition} \label{prop:evalProps}
 We have the following
\begin{enumerate}
\item $0<\alpha_1 < \beta_1<\alpha_2<\ldots<\beta_n<\alpha_{n+1}<\ldots$
\item $f'(\mu)<0$ and $f(\mu)>0$ on $[0,\alpha_1)$  and $(\beta_n,\alpha_{n+1})$ for $n\geq1$.
\end{enumerate}
\end{proposition}

\begin{proof}

1. Let $\kappa = 1-\rb/R$ and $x=R\sqrt{\mu}$. We make the change of variable in $f(\mu)$ and obtain 
\begin{equation*}
f(\mu)\equiv\tilde{f}(x) =\frac{R\paren{x-\tan( \kappa x)}}{x^2\tan(\kappa x)+x}=\frac{R\paren{1-\frac{\tan(\kappa x)}{x}}}{x\tan(\kappa x)+1}\equiv\frac{N(x)}{D(x)}.
\end{equation*}
Let $d_n$ be such that $N(d_n) =0$ and $\eta_n$ be such that
$D(\eta_n)=0$. In terms of the old variable, we have $d_n
=R\sqrt{\alpha_n}$ and $\eta_n=R\sqrt{\beta_n}$. $N(x) =0$ imply that
$\tan(\kappa x) = x$ and $D(x)=0$ imply that $\tan(\kappa x)
=-\frac{1}{x}$. Note that the functions $\tan(\kappa x)$, $x$ and
$-\frac{1}{x}$ are all monotone increasing. Finally if we let
$\theta_n =\frac{\pi}{2\kappa}(2n-1)$ be the vertical asymptotes of
$\tan(\kappa x)$ one easily checks that we have
\begin{equation*}
0<d_1<\theta_1<\eta_1<\ldots<d_n<\theta_n<\eta_n\ldots
\end{equation*}
This proves 1.

2. $N>0$ on $\displaystyle\bigcup_{n\geq
    1}(\theta_n,d_{n+1})\cup(0,d_1)$, and $N<0$ on
  $\displaystyle\bigcup_{n\geq 1}(d_n,\theta_n)$. Similarly,
  $D>0$ on $\displaystyle\bigcup_{n\geq
    1}(\eta_n,\theta_{n+1})\cup(0,\theta_1)$ and $D<0$ on
  $\displaystyle\bigcup_{n\geq 1}(\theta_n,\eta_n)$. Thus it follows
  that $\tilde{f}(x) >0$ if and only if $x\in
  \displaystyle\bigcup_{n\geq 1}(\eta_n,d_{n+1})\cup (0,d_1)$.  Next,
  we show that $f'<0$. Note
\begin{align*}
N'(x) = \frac{R\paren{\tan(\kappa x) - \kappa x\sec^2(\kappa x)}}{x^2}
= \frac{R\paren{\sin(2\kappa x)-2\kappa x}}{2x^2\cos^2(\kappa x)},
\end{align*}
and 
\begin{align*}
  D'(x) = \tan(\kappa x) + \kappa x\sec^2(\kappa x)
  =\frac{\sin(2\kappa x)+2\kappa x}{2\cos^2(\kappa x)}.
\end{align*}
Since $\abs{\sin(\theta)} \leq\abs{\theta}$ it follows that
$N'(x)\leq0$ and that $D'(x)\geq 0$, with equality in both only when
$x=0$.  From what we have shown above, $\tilde{f}(x)> 0$ if and only
if both $N(x)>0$ and $D(x)>0$ so that
\begin{align*}
f'(\mu) = \tilde{f}'(x)\cdot \D{x}{\mu}
= \brac{\frac{D(x)N'(x)-N(x)D'(x)}{{D(x)}^2}}\cdot \D{x}{\mu}
 < 0.
\end{align*}  
\end{proof}

\begin{proposition}
  The Doi eigenvalues $\mu_n(\lambda)$ satisfy for $n$ such that
  $\mu_n(\lambda)\leq\lambda$
\begin{equation*}
0<\mu_1(\lambda)<\alpha_1 < \beta_1<\mu_2(\lambda)<\alpha_2<\ldots<\beta_{n-1}<\mu_{n}(\lambda)<\alpha_{n}
\end{equation*}
\end{proposition}

\begin{proof}
  This follows from the fact that $f(\mu)$ is decreasing wherever it
  is positive and that $A(\mu;\lambda)$ is positive and increasing for
  $\mu \leq \lambda$.
\end{proof}

We will also have need for the following\\
 \begin{proposition}\label{prop:eigenBehaviour}
   Let $\{\gamma_n\}$ denote the eigenvalues for the (positive)
   radially symmetric Laplacian on $\left[0, R\right)$, with zero Neumann boundary conditions on
   the ball of radius $R$. Let $\alpha_n$ and $\mu_n$ be as above.
   Then the following hold:
\begin{enumerate}
\item The $\mu_n(\lambda)$ are continuous and monotone increasing in $\lambda$ for all $n\geq1$
\item For all $n\geq 1$ and any fixed $\lambda$, we have that 
\begin{equation*}
\mu_n(0) = \gamma_n =\paren{\frac{(n-1)\pi}{R}}^2\leq \mu_n(\lambda)
\end{equation*}

\end{enumerate}
\end{proposition}

\begin{remark}
  The proof is a straightforward application of the variational
  minimax principle of Poincare.
  We do not show  it here.\\
\end{remark}

\begin{proposition}\label{prop:Cardinality_bound}
  For any $L \in \R_{+}$, define the index set $\mathcal{A}(L) = \{n|
  \alpha_n\leq L\}$. If we let $\abs{\mathcal{A}(L)} \equiv
  \card(\mathcal{A}(L))$ then given $\delta>0$ there exists constants
  $C_1^{*}(\delta)$ and $C_2^{*}(\delta)$ such that for $L\geq \delta$
\begin{equation}
C_1^{*}(\delta)\sqrt{L} \leq\abs{ \mathcal{A}(L)} \leq C_2^{*}(\delta)\sqrt{L}
\end{equation}
\end{proposition}

\begin{proof}
  Write $\til{L} =R\sqrt{L}$ and again define $\kappa = 1 - \rb / R$.
  Then $\abs{ \mathcal{A}(L)}$ is just the number of solutions to
  $\tan(\kappa x)= x$ which lie in the interval $[0,\til{L}]$. This
  number is well approximated by the number of vertical asymptotes of
  $\tan(\kappa x)$. It then follows that
\begin{equation*}
\frac{\kappa \til{L}}{\pi}-1 \leq\abs{ \mathcal{A}(L)} \leq \frac{\kappa \til{L}}{\pi}+1
\end{equation*}
so that 
\begin{equation*}
\sqrt{L}\paren{\frac{R-\rb}{\pi} -\frac{1}{\sqrt{L}}} \leq\abs{ \mathcal{A}(L)} \leq\sqrt{L}\paren{\frac{R-\rb}{\pi} +\frac{1}{\sqrt{L}}}. 
\end{equation*}
If $L\geq\delta$ the choice $\displaystyle C_1^{*}(\delta) = \frac{R-\rb}{\pi} -\frac{1}{\sqrt{\delta}}$ and $\displaystyle C_2^{*}(\delta) = \frac{R-\rb}{\pi} +\frac{1}{\sqrt{\delta}}$ gives the proposition.\\
\end{proof}

\begin{remark}\label{rem:card_choice}
  In practice we will choose $\delta =4\pi^2/(R-\rb)^2$ so that
  $C_1^{*}=(R-\rb)/2\pi$.
\end{remark}

We now give our main convergence estimate for the eigenvalues of the
Doi model. The following theorem can be regarded as the heart of the
subsequent computations.
\begin{theorem}\label{thm:EigenConv}
  Let $0<\sigma_0 < \tfrac{1}{4}$ and define $M(\lambda) \equiv
  K_0\lambda^{\sigma}$ for $K_0 > 1$. For any fixed $\sigma \in
  (0,\sigma_0]$ there exists $\lambda_0 >0$ such that for
  $\lambda\geq\lambda_0$, $\frac{M(\lambda)}{\lambda}\leq
  \frac{1}{2}$. Then for $\alpha_n \leq M(\lambda)$,
  $\mu_n(\lambda) \to \alpha_n$,
  $\mathcal{O}(\lambda^{-(\frac{1}{2}-2\sigma)})$. \\
\end{theorem}

\begin{remark}
  Note that in the remainder $C$ will denote an arbitrary constant
  that may depend on $R$, $\rb$, and $\lambda_0$. We will also
  subsequently assume $\lambda_0 > 1$.
\end{remark}

\begin{proof}
  Recall that the Doi eigenvalues $\mu_n(\lambda)$ satisfy
\begin{equation*}
f(\mu_{n}(\lambda)) = A(\mu_{n}(\lambda),\lambda).
\end{equation*}
As before let $\kappa = 1- \rb / R$ and let $x=R\sqrt{\mu}$. Recalling
the definitions of $D(x)$ and $N(x)$ from
Proposition~\ref{prop:evalProps}, define
\begin{equation*}
B(x;\lambda) = D(x)\tilde{A}(x,\lambda)\equiv  \frac{\brac{x\tan(\kappa x)+1}\tanh\paren{\rb\sqrt{1-\frac{x^2}{\lambda R^2}}}}{\sqrt{\lambda}\sqrt{1-\frac{x^2}{\lambda R^2}}}.
\end{equation*}
It follows that the rescaled Doi eigenvalues $x_n(\lambda)$ satisfy
\begin{equation*}
N(x_n(\lambda)) = B(x_n(\lambda); \lambda).
\end{equation*}

For $K_0 > 1$ we choose $\lambda\geq(2K_0)^\frac{1}{1-\sigma_0}$ so
that $\frac{M(\lambda)}{\lambda} \leq \frac{1}{2}$.  Recall
$f(\mu)\equiv\tilde{f}(x)$.  We restrict to $\{x : \frac{x^2}{R^2}\leq
M(\lambda), \tilde{f}(x) \geq 0\}$, and let $h(u) =(1-u)^{-1/2}$.
Since $h$ is monotone it follows that
\begin{equation*}
\frac{1}{\sqrt{\lambda}\sqrt{1-\frac{x^2}{\lambda R^2}}} \equiv \frac{h(\tfrac{x^2}{R^2\lambda})}{\sqrt{\lambda}}\leq \frac{h(\tfrac{1}{2})}{\sqrt{\lambda}} =\frac{\sqrt{2}}{\sqrt{\lambda}}.
\end{equation*}
As shown in Proposition~\ref{prop:evalProps}, $\tilde{f}(x) \geq 0$
implies $-\frac{1}{x} \leq \tan(\kappa x)\leq x$ so that
\begin{align*}
\abs{B(x;\lambda)} \leq \frac{\brac{1 + x\tan(\kappa x)}\sqrt{2}}{\sqrt{\lambda}}
\leq  \frac{\brac{1 + x^2}\sqrt{2}}{\sqrt{\lambda}}.
\end{align*}

Write $x_n(\lambda) = d_n -\epsilon_n$. (Recall $d_n =
R\sqrt{\alpha_n}$ and $\eta_n = R\sqrt{\beta_n}$.) Then
\begin{equation*}
N(d_n - \epsilon_n) = B(x_n(\lambda);\lambda).
\end{equation*}
Applying the mean value theorem we get, for some $e_n \in (x_n(\lambda), d_n)$, that
\begin{equation*}
N(d_n) -N'(e_n)\epsilon_n = B(x_n(\lambda);\lambda).
\end{equation*}
As $N(d_n) = 0$ and $N'(e_n) < 0$, 
\begin{equation*}
\abs{N'(e_n)}\epsilon_n \leq \frac{\brac{1 + x_n^2}\sqrt{2}}{\sqrt{\lambda}}
\end{equation*}
which gives that 
\begin{align} \label{eq:epsnEstimate}
\epsilon_n
&\leq \frac{\sqrt{2}\brac{1+x^2_n}\paren{2e_n^2\cos^2(\kappa e_n)}}{R\sqrt{\lambda}\paren{2\kappa e_n -\sin(2\kappa e_n)}}
 \leq \frac{\sqrt{2}\brac{1+d^2_n}e_n}{\kappa R\sqrt{\lambda}\paren{1 -\tfrac{\sin(2\kappa e_n)}{2\kappa e_n }}}
\leq \frac{C\brac{1+d^2_n}d_n}{\sqrt{\lambda}\paren{1 -\tfrac{\sin(2\kappa e_n)}{2\kappa e_n }}}
\end{align}

 For any fixed $c \in \paren{0, \kappa \, x_1(1)}$,
monotonicity of the eigenvalues implies $0<\frac{c}{2\kappa}\leq
x_1(\lambda_0)\leq x_n(\lambda) <e_n <d_n$, and therefore 
$c<2\kappa e_n$.  Define $\displaystyle l(\theta) =
1-\frac{\sin(\theta)}{\theta}$. It follows easily that for $\theta\geq
c$ there exists $0<m<1$ such that
\begin{equation*}
m \leq l(\theta)\leq 1
\end{equation*}

Using this bound in~\eqref{eq:epsnEstimate}, in the original unscaled
variables we find that
\begin{align} \label{eq:eigenvalSqrtEst}
R(\sqrt{\alpha_n} -\sqrt{\mu_n})&\leq \frac{CR\sqrt{\alpha_n}\brac{1+R^2\alpha_n}}{\sqrt{\lambda}}
\end{align}
which implies
\begin{align*}
\alpha_n-\mu_n & \leq \frac{2C\alpha_n\brac{1+R^2\alpha_n}}{\sqrt{\lambda}}.
\end{align*}
For $\alpha_n\leq M(\lambda) \equiv K_0\lambda^{\sigma}$, $M(\lambda) > 1$
implies
\begin{align*}
\alpha_n-\mu_n & \leq \frac{C(M(\lambda))^2}{\sqrt{\lambda}}\equiv \frac{CK_0^2}{\lambda^{\frac{1}{2}-2\sigma}}.
\end{align*}
\end{proof}

\begin{remark}
Of interest is the possibility of tighter estimates here. In fact, one can show that if $f''(\mu) >0$ wherever $f(\mu) >0$ we actually have
\begin{equation*}
\alpha_n-\mu_n  \leq \frac{C\alpha_n}{\lambda^{\frac{1}{2}}}.
\end{equation*}
\end{remark}

Theorem~\ref{thm:EigenConv} and Proposition~\ref{prop:eigenBehaviour}
immediately implies
\begin{corollary} \label{cor:eigenMonotone}
For any fixed $n$, we have that $\mu_n(\lambda)$ converges monotonically to $\alpha_n$ as $\lambda\to\infty$.
\end{corollary}


\subsection{Eigenfunction Estimates}\label{sec:eigenfunc_estim}
In this subsection we carry over the estimates for the eigenvalues
obtained in the last subsection to obtain the uniform convergence in
$r$ of the eigenfunctions as $\lambda \to \infty$. Though unstated, in
the remainder all theorems and lemmas include the assumptions of
Theorem~\ref{thm:EigenConv}.
\begin{lemma}\label{lem:eigen_func_1}
  The (unnormalized) Doi and Smoluchowski eigenfunctions satisfy
\begin{equation*}\label{eq:eigenfunc_asym}
  \sup_{r\in[\rb,R]}{\abs{\phi_n(r)-\psi^{out}_n(r)}} =\mathcal{O}\paren{\lambda^{-(\frac{1}{2} -\frac{3\sigma}{2})}}
\end{equation*} for $\mu_n< M(\lambda) = K_0 \lambda^\sigma$ as $\lambda\to\infty$.
\end{lemma}

\begin{proof}
\begin{align*}
\abs{\phi_n(r)-\psi^{out}_n(r)} 
&\leq \frac{1}{r}\abs{{\frac{\sin(\sqrt{\alpha_n}(R-r))}{R\sqrt{\alpha_n}}-\frac{\sin(\sqrt{\mu_n}(R-r))}{R\sqrt{\mu_n}}}} + \frac{1}{r}\bigg| \cos(\sqrt{\mu_n}(R-r))-\cos(\sqrt{\alpha_n}(R-r)) \bigg| \\
&:= I + II
\end{align*}
Note that 
\begin{align*}
I&\leq\frac{1}{\rb}\abs{\frac{\sin(\sqrt{\alpha_n}(R-r))}{R\sqrt{\alpha_n}}-\frac{\sin(\sqrt{\mu_n}(R-r))}{R\sqrt{\alpha_n}}} +\frac{1}{\rb}\abs{\frac{\sin(\sqrt{\mu_n}(R-r))}{R\sqrt{\alpha_n}} - \frac{\sin(\sqrt{\mu_n}(R-r))}{R\sqrt{\mu_n}}}\\
&:= I_a + I_b.
\end{align*}
We find
\begin{align*}
I_a &\leq \frac{1}{R\rb\sqrt{\alpha_n}}\abs{2\sin\paren{(R-r)\frac{\sqrt{\alpha_n} -\sqrt{\mu_n}}{2}}\cos\paren{(R-r)\frac{\sqrt{\alpha_n} +\sqrt{\mu_n}}{2}}}
\leq\frac{(R-\rb)\paren{\sqrt{\alpha_n} -\sqrt{\mu_n}}}{R\rb\sqrt{\alpha_n}}.
\end{align*}
Similarly we have 
\begin{align*}
I_b &\leq \frac{\abs{\sin(\sqrt{\mu_n}(R-r))}}{R\rb}\brac{\frac{1}{\sqrt{\mu_n}} - \frac{1}{\sqrt{\alpha_n}}}
\leq\frac{(R-\rb)}{R\rb}\brac{\frac{\sqrt{\alpha_n} -\sqrt{\mu_n}}{\sqrt{\alpha_n}}}.
\end{align*}
Combining these with~\eqref{eq:eigenvalSqrtEst} 
\begin{align*}
  I  &\leq\frac{C\brac{1+R^2\alpha_n}}{\sqrt{\lambda}}
\leq \frac{CM(\lambda)}{\sqrt{\lambda}} = \frac{CK_0}{\lambda^{\frac{1}{2}-\sigma}}.
\end{align*}
For $II$ we have
\begin{align*}
II &\leq\frac{1}{\rb}\brac{2\sin\paren{(R-r)\frac{\sqrt{\alpha_n} -\sqrt{\mu_n}}{2}}\sin\paren{(R-r)\frac{\sqrt{\alpha_n} +\sqrt{\mu_n}}{2}}}\\
&\leq\frac{R-\rb}{\rb}(\sqrt{\alpha_n} -\sqrt{\mu_n})\\
&\leq \frac{C(R-\rb)\sqrt{\alpha_n}\brac{1+R^2\alpha_n}}{\rb \sqrt{\lambda}}\\
&\leq \frac{C\paren{M(\lambda)}^{\frac{3}{2}}}{\sqrt{\lambda}} = \frac{CK_0^{\frac{3}{2}}}{\lambda^{\frac{1}{2}-\frac{3\sigma}{2}}}.
\end{align*}
It follows that 
\begin{equation}\label{eqn:eigenfunc_Diff}
\abs{\phi_n(r)-\psi^{out}_n(r)} 
\leq \frac{CK_0}{\lambda^{\frac{1}{2} -\frac{3\sigma}{2}}}\brac{\frac{1}{\lambda^{\sigma/2}}+\sqrt{K_0}} 
\leq \frac{C_1(\rb,R,K_0)}{\lambda^{\frac{1}{2} -\frac{3\sigma}{2}}}
\end{equation}
This concludes the proof.
\end{proof}

We now prove several uniform properties of the eigenfunctions we will
use in the next subsection.
\begin{lemma}\label{lem:eigen_func_2}
\begin{enumerate}
\item There exist a $\lambda_0$, $C_2 =C_2(\rb,R)$, such that for all $\lambda \geq \lambda_0$ and $n \in \Z^{+}$
\begin{equation}\label{eqn:sup_eigenfunction_bound}
\max \paren{\sup_{r\in[\rb,R]}\abs{\psi_n^{out}(r)}, \sup_{r\in[\rb,R]}\abs{\phi_n(r)}} \leq C_2.
\end{equation}
\item Let $b_n = \norm{\psi_n}^{-2}_2$ and $a_n =\norm{\phi_n}^{-2}_2$. Then there exists $C_3$ such that
for all $\lambda \geq \lambda_{0}$
\begin{align}\label{eqn:sup_normalization_bound}
\max \Big( \sup_{n} \{a_n\}, \,  \sup_{n} \{b_n\}  \Big) \leq C_3.
  \end{align}
\end{enumerate}
\end{lemma}

\begin{proof}
1. We start by defining for $z\geq0$, and $\rb\leq r\leq R$ the auxiliary function
\begin{equation*}
H(z,r) := \frac{1}{r}\brac{\frac{\sin(\sqrt{z}(R-r))}{R\sqrt{z}}-\cos(\sqrt{z}(R-r))}.
\end{equation*}
Note $\phi_n(r)\equiv H(\alpha_n,r)$ and $\psi_n^{out}(r)\equiv H(\mu_n,r)$. Now for $z \geq z_0 >0$ we have that
 \begin{align*}
 \abs{H(z,r)} &= \frac{1}{r}\abs{\brac{\frac{\sin(\sqrt{z}(R-r))}{R\sqrt{z}}-\cos(\sqrt{z}(R-r))}}\\
&\leq \frac{1}{\rb}\brac{\frac{1}{R\sqrt{z}} +1}\\
 &\leq \frac{1}{\rb}\brac{\frac{1}{R\sqrt{z_0}} +1} =:C_2(z_0).
 \end{align*}
 By Corollary~\ref{cor:eigenMonotone} there exists $\lambda_0$ such
 that for $\lambda \geq\lambda_0$
\begin{equation*}
\alpha_1\geq\mu_1(\lambda)\geq\mu_1(\lambda_0)\geq\frac{\alpha_1}{2}>0.
\end{equation*}
Note that we are using the fact that both the Doi and Smoluchowski
eigenvalues can be written in non-decreasing order. Choosing $z_0
=\alpha_1/2$, proves the first part of the lemma.

2. To prove the second part start by defining
\begin{equation*}
h(z):=  \int_{\rb}^{R}{\paren{H(z,r)}^2 r^2\,dr}.
\end{equation*}
Once again we have that $\norm{\psi^{out}_n}^{2}_2\equiv h(\mu_n)$ and
$\norm{\phi_n}^{2}_2 \equiv h(\alpha_n)$. \emph{A priori} we have that
$h(z)>0$ for all $z\geq0$.  An explicit computation shows that
for $z>0$, $h(z)$ is continuous,  $\lim\limits_{z\to\infty} h(z) =
(R-\rb)/2 >0$, and 
$\lim\limits_{z\to0} h(z) = (R^3 - \rb^3) / 3 R^2$. 
With the positivity of $h(z)$ on $\left[0,\infty\right)$, these
results imply that $A:=\inf h(z) >
0$. 
It then follows that
\begin{equation*}
a_n =\frac{1}{\norm{\phi_n}^{2}_2} \equiv \frac{1}{h(\alpha_n)}\leq\frac{1}{A}=:C_3
\end{equation*}
and 
\begin{equation*}
b_n =\frac{1}{\norm{\psi_n}^{2}_2}\leq \frac{1}{\norm{\psi^{out}_n}^{2}_2}\equiv \frac{1}{h(\mu_n)}\leq\frac{1}{A}=:C_3.
\end{equation*}
This concludes the proof of the lemma.
\end{proof}



These results imply that 
\begin{lemma} \label{lem:eigen_func_3}
There exists $C_4$ such that for $n$ with $\mu_n(\lambda)\leq\alpha_n\leq M(\lambda)$
\begin{equation*}
\abs{b_n - a_n}\leq  \frac{C_4}{\lambda^{\frac{1}{2} -\frac{3\sigma}{2}}}.
\end{equation*}
\end{lemma}

\begin{proof}
 We start by noting that
 \begin{align*}\label{eqn:B_n_1}
 \abs{\norm{\psi^{out}_n}^2_2 -\norm{\phi_n}^2_2 }&=\abs{\int_{\rb}^{R}\paren{(\psi^{out}_n)^2 -\phi_n^2}r^2\,dr}\nonumber\\
 &\leq \int_{\rb}^{R}\abs{\psi^{out}_n -\phi_n}\abs{\phi_n +\psi^{out}_n}r^2\,dr\nonumber\\
 &\leq \frac{C_1(\rb,R,K_0)}{\lambda^{\frac{1}{2} -\frac{3\sigma}{2}}}\brac{\sup_{r \in \brac{\rb,R}} {\paren{\abs{\psi^{out}_n}+\abs{\phi_n}}}}\int_{\rb}^{R}{r^2\,dr} =\frac{2C_1C_2(R^3-\rb^3)}{3\lambda^{\frac{1}{2} -\frac{\sigma}{2}}}.
 \end{align*}

 To get the last line, we have used Lemmas~\ref{lem:eigen_func_1}
 and~\ref{lem:eigen_func_2}. A direct computation shows that
\begin{align*}
 \abs{\norm{\psi^{out}_n}^2_2 -\norm{\psi_n}^2_2 } = \norm{\psi^{in}_n}^2_2 &= \paren{\frac{\tfrac{1}{R\sqrt{\mu_n}}\sin(\sqrt{\mu_n}(R-\rb))-\cos(\sqrt{\mu_n}(R-\rb))}{\sinh(\rb\sqrt{\lambda-\mu_n})}}^2\int_0^{\rb}{\sinh^{2}(r\sqrt{\lambda-\mu_n})} \, r^2  dr\\
&\leq \frac{C\paren{\sinh(2\rb\sqrt{\lambda-\mu_n}) -2\rb\sqrt{\lambda-\mu_n}}}{\sqrt{\lambda-\mu_n}\sinh^{2}(\rb\sqrt{\lambda-\mu_n})} \leq \frac{C}{\lambda^{\frac{1}{2}}}.
\end{align*} 
(Here we have used that $\sqrt{\lambda - \mu_n} \geq \sqrt{\lambda -
  M(\lambda)} \geq \sqrt{\frac{\lambda}{2}}$.)  Using the preceding
bounds we find, that
\begin{align*}
 \abs{b_n -a_n} = \abs{\frac{1}{\norm{\psi_n}^2_2} -\frac{1}{\norm{\phi_n}^2_2}} 
 = \frac{\abs{\norm{\phi_n}^2_2 -\norm{\psi_n}^2_2 }}{\norm{\phi_n}^2_2\norm{\psi_n}^2_2}
&\leq \frac{\abs{\norm{\psi^{out}_n}^2_2 -\norm{\psi_n}^2_2 }}{\norm{\psi_n}^2_2\norm{\phi_n}^2_2}+\frac{\abs{\norm{\psi^{out}_n}^2_2 -\norm{\phi_n}^2_2 }}{\norm{\psi_n}^2_2\norm{\phi_n}^2_2}\\
& \leq\frac{CC_3^2}{\lambda^{\frac{1}{2}}} + \frac{2C_1C_2C_3^2(R^3-\rb^3)}{3\lambda^{\frac{1}{2} -\frac{3\sigma}{2}}}\\
& =\frac{1}{\lambda^{\frac{1}{2} -\frac{3\sigma}{2}}}\brac{\frac{CC_3^2}{\lambda^{\frac{3\sigma}{2}}} + \frac{2}{3}C_1C_2C_3^2(R^3-\rb^3)} .
 \end{align*}
 The choice $C_4 = CC_3^2 + 2C_1C_2C_3^2(R^3-\rb^3) / 3$ gives the lemma.
\end{proof}


\subsection{An Error Estimate for the Convergence of the Doi to
  Smoluchowski Model}\label{sec:main_result}

We now show the uniform convergence of the Green's function of the
radially symmetric Doi PDE~\eqref{eq:radpureDoi} to the Green's
function of the radially symmetric Smoluchowski
PDE~\eqref{eq:radpureSmol} model. The error bound we give shows that
the convergence of the Doi model to the Smoluchowski model can not be
expected to be faster than $O(\lambda^{-1/2})$ as $\lambda \to
\infty$.

\begin{theorem} \label{thm:densityConv}
  Fix $0<\sigma_0<\frac{1}{4}$ and let the initial condition $g(r)
  =\delta(r-r_0)/r^2$, where $r_0\in(\rb,R)$ . For $ t\geq\delta>0$,
  there exists a function $u(t)$ and $\lambda_0$ such that for all
  $\lambda\geq\lambda_0$ we have:
  \begin{equation}
    \sup_{r\in[\rb,R]}{\abs{\rho(r,t)-\pl(r,t)}}\leq \frac{u(t)}{\lambda^{\frac{1}{2}-2\sigma}}.
  \end{equation}
  Moreover, for $t \in \left[ \delta, \infty \right)$, $u(t)$ is
  uniformly bounded.
\end{theorem}

\begin{proof}
  The main idea is to use the series representation of the solutions
  to both models to estimate the error. There will be a proliferation
  of constants which we shall repeatedly and unceremoniously denote by
  $C$. For $r \in \paren{\rb,R}$, we begin by writing:
  \begin{align*}
    \pl(r,t)-\rho(r,t)&=\sum_{\{n \mid \alpha_n<M(\lambda)\}}{b_n\psi_{n}(r_0)\psi^{out}_{n}(r)e^{-\mu_{n}(\lambda)t}-a_n\psi_{n}(r_0)\phi_{n}(r)e^{-\alpha_{n}t}} \\
    &\phantom{=} + \sum_{\{n \mid \alpha_n\geq M(\lambda)\}}{b_n\psi_{n}(r_0)\psi^{out}_{n}(r)e^{-\mu_{n}(\lambda)t}-a_n\psi_{n}(r_0)\phi_{n}(r)e^{-\alpha_{n}t}}\\
    &:=I +II.
  \end{align*}
  We deal with the finitely indexed sum, $I$, first. Define the index
  set $\mathcal{A}_\lambda = \{n \mid \alpha_n<M(\lambda)\}$. From now on,
  for simplicity of presentation, we write $\psi_n$ for
  $\psi_n^{out}$.  Let $I = I_a +I_b + I_c +I_d$, where
  \begin{align*}
    I_{a} &= \sum_{\mathcal{A}_\lambda} b_n\paren{\psi_{n}(r_0)-\phi_{n}(r_0)}\psi_{n}(r)e^{-\mu_{n}(\lambda)t},
    &I_{b} &= \sum_{\mathcal{A}_\lambda} b_n\phi_{n}(r_0) \paren{ \psi_{n}(r) - \phi_{n}(r)}e^{-\mu_{n}(\lambda)t}, \\
    I_{c} &= \sum_{\mathcal{A}_\lambda} \paren{b_n-a_{n}} \phi_{n}(r_0)\phi_{n}(r)e^{-\mu_{n}(\lambda)t}, 
    &I_{d} &= \sum_{\mathcal{A}_\lambda} a_n \phi_{n}(r_0) \phi_{n}(r) \paren{e^{-\mu_{n}(\lambda)t} - e^{-\alpha_{n}t}}. 
  \end{align*}
  Recalling that $\gamma_{n}$ denotes the $n$th eigenvalue of the
  radically symmetric Laplacian on $\left[0,R \right)$ with a zero
  Neumann boundary condition at $R$ (see
  Proposition~\ref{prop:eigenBehaviour}), we find
  \begin{align*}
    \abs{I_a} 
    &\leq \sum_{\mathcal{A}_\lambda}{\abs{b_n}\abs{\psi_n(r)}\abs{\psi_n(r_0)-\phi_n(r_0)}e^{-\mu_{n}t}}
    \leq \frac{C}{\lambda^{\frac{1}{2} -\frac{3\sigma}{2}}}\sum_{n=1}^{\infty}{e^{-\gamma_{n}t}}.
  \end{align*}
  Here we have applied Proposition~\ref{prop:eigenBehaviour},
  Lemma~\ref{lem:eigen_func_1}, and Lemma~\ref{lem:eigen_func_2} (in
  particular~\eqref{eqn:eigenfunc_Diff},
  \eqref{eqn:sup_eigenfunction_bound}, and
  \eqref{eqn:sup_normalization_bound}).
  The same argument shows
  \begin{align*}
    \abs{I_b} 
    &\leq \sum_{\mathcal{A}_\lambda}{\abs{b_n}\abs{\psi_n(r)-\phi_n(r)}\abs{\phi_n(r_0)}e^{-\mu_{n}t}}
    \leq \frac{C}{\lambda^{\frac{1}{2} -\frac{3\sigma}{2}}}\sum_{n=1}^{\infty}{e^{-\gamma_{n}t}},
  \end{align*}
  and using Lemma~\ref{lem:eigen_func_3} too we find
  \begin{align*}  
    \abs{I_c} 
    &\leq\sum_{\mathcal{A}_\lambda}{\abs{b_n-a_n}\abs{\phi_{n}(r_0)}\abs{\phi_{n}(r)}e^{-\mu_{n}t}}
    \leq \frac{C}{\lambda^{\frac{1}{2} -\frac{3\sigma}{2}}}\sum_{n=1}^{\infty}{e^{-\gamma_{n}t}}.
  \end{align*}
  Finally, we have that
  \begin{align*}
    \abs{I_d} 
    &\leq \sum_{\mathcal{A}_{\lambda}}{\abs{a_n}\abs{\phi_n(r)}\abs{\phi_n(r_0)}\abs{1-e^{-(\alpha_{n}-\mu_n)t}}} \, e^{-\mu_{n}t}.
  \end{align*}
  For $s\geq 0$, $1-e^{-s}\leq \abs{s}$ so that, using the same lemmas
  as before and Theorem~\ref{thm:EigenConv},
  \begin{align*}
    \abs{I_d}&\leq C \sum_{\mathcal{A}_{\lambda}}{e^{-\mu_{n}t}\abs{\alpha_n -\mu_n}t} 
    \leq \frac{C}{\lambda^{\frac{1}{2} -2\sigma}}\sum_{n=1}^{\infty}{t \, e^{-\mu_{n}t}}.
  \end{align*}

  We now bound the tail of the series, $II$. First define the index
  set $\mathcal{B}_\lambda = \{n \mid \alpha_n\geq M(\lambda)\}$. We now
  specify the choice $K_0 > 4\pi^2/(R-\rb)^2$ which guarantees that
  $K_0\lambda^\sigma>4\pi^2/(R-\rb)^2$ (see
  Remark~\ref{rem:card_choice} after the proof of
  Proposition~\ref{prop:Cardinality_bound}). Using the uniform bounds
  on $\psi_n$, $\phi_n$, $a_n$, and $b_n$ and
  Proposition~\ref{prop:Cardinality_bound} we find
  \begin{align*}
    \abs{II} & \leq C_1 \sum_{n\geq C_1^{*}\sqrt{K_0}\lambda^{\frac{\sigma}{2}}} e^{-\mu_{n}t} + C_2 \sum_{n\geq C_1^{*}\sqrt{K_0}\lambda^{\frac{\sigma}{2}}}{e^{-\alpha_{n}t}}
    \leq C\sum_{n\geq\lambda^{\frac{\sigma}{2}}}{e^{-\gamma_{n}t}}.
  \end{align*}
  We thus obtain the error estimate that
  \begin{equation}\label{eq:diff_sum}
    \abs{\rho(r,t)-\pl(r,t)}\leq C \brac{\frac{1}{\lambda^{\frac{1}{2} -\frac{3\sigma}{2}}}\sum_{n=1}^{\infty}{e^{-\gamma_{n}t}}
      + \frac{1}{\lambda^{\frac{1}{2} -2\sigma}}\sum_{n=1}^{\infty}{t \, e^{-\mu_{n}t}}
      + \sum_{n\geq\lambda^{\frac{\sigma}{2}}}{e^{-\gamma_{n}t}}}.
  \end{equation}
  
  We estimate the terms in \eqref{eq:diff_sum} one at a time.  First
  \begin{align*}
    \sum_{n=1}^{\infty}{e^{-\gamma_{n}t}}& = \sum_{n=0}^{\infty}{\exp\brac{-\frac{n^2t\pi^2}{R^2}}}
    \leq 1 + \int_{0}^{\infty}{\exp\brac{-\frac{x^2t\pi^2}{R^2}}\,dx}
    = 1 + \frac{R}{\sqrt{4\pi t}},
  \end{align*} 
  while
  \begin{align*}
    \sum_{n=1}^{\infty}{te^{-\mu_{n}t}}& \leq te^{-\mu_1(\lambda)t} +  \sum_{n=2}^{\infty}{te^{-\gamma_{n}t}} \\
    &\leq te^{-\mu_1(\lambda_0)t} +  \sum_{n=1}^{\infty}{t\exp\brac{-\frac{n^2t\pi^2}{R^2}}} \\
    &\leq  \frac{1}{e\mu_1(\lambda_0)} +\frac{R^2}{e\pi^2}\sum_{n=1}^{\infty}{\frac{1}{n^2}}\leq \bar{C}.
  \end{align*}
  Finally, we bound the third term in~\eqref{eq:diff_sum}:
  \begin{align*}
    \sum_{n\geq\lambda^{\frac{\sigma}{2}}}{e^{-\gamma_{n}t}} &=  \sum_{n\geq\lambda^{\frac{\sigma}{2}}-1}{\exp\brac{-\frac{n^2t\pi^2}{R^2}}}\\
    &\leq  \int_{\lambda^{\frac{\sigma}{2}}-2}^{\infty}{\exp\brac{-\frac{x^2t\pi^2}{R^2}}\,dx}\\
    &\leq  \frac{R}{\sqrt{4\pi t}}\erfc\paren{\frac{(\lambda^{\sigma/2} -2)\sqrt{t}\pi}{R}}\\
    &\leq \frac{\hat{C} R}{\sqrt{4\pi t}}e^{-\til{C}\lambda^{\sigma}t}.
  \end{align*}
  Combining the preceding estimates we have
  \begin{align*}\label{eq:diff_sum_2}
    \abs{\rho(r,t)-\pl(r,t)}&\leq\frac{C}{\lambda^{\frac{1}{2} -2\sigma}}
    \brac{\frac{1}{\lambda^{\sigma/2}} \paren{1 + \frac{R}{\sqrt{4\pi t}}} + \bar{C} + \hat{C} \lambda^{\frac{1}{2}-2\sigma} \frac{R}{\sqrt{4\pi t}}  e^{-\til{C}\lambda^{\sigma}t}}\\
    &\leq \frac{C}{\lambda^{\frac{1}{2} -2\sigma}}\brac{\frac{1}{\sqrt{t}}\paren{1+\frac{C(\sigma)}{t^{\frac{1}{2\sigma}-2}}}+1}.
  \end{align*} 
  Here we absorbed the maximum of the many constants into $C$ and 
  \begin{equation*}
    C(\sigma) \equiv \paren{\frac{1}{\til{C}e}\brac{{\frac{1}{2\sigma}-2}}}^{{\frac{1}{2\sigma}-2}}.
  \end{equation*}
  Let
  \begin{equation*}
    u(t):= C\brac{\frac{1}{\sqrt{t}}\paren{1+\frac{C(\sigma)}{t^{\frac{1}{2\sigma}-2}}}+1}.
  \end{equation*}
  For $t \geq \delta > 0$ we see that $u(t)$ is uniformly bounded,
  concluding the proof.
\end{proof}
\begin{remark}
  Note that even for $t\geq\delta>0$, $u(t)\to\infty$ as $\sigma\to 0$
  because $u(t)$ blows up like
  $\paren{\frac{1}{\sqrt{\delta\sigma}}}^{\frac{1}{\sigma}}$.
\end{remark}

\section{Conclusion}
We have shown for the special case of two molecules that may undergo
the annihilation reaction, $\textrm{A} + \textrm{B} \to \varnothing$,
with one of the two molecules stationary, the solution to the Doi
model~\eqref{eq:radpureDoiAvg} converges to the solution of the
Smoluchowski model~\eqref{eq:radpureSmolAvg} as $\lambda \to \infty$.
A rigorous asymptotic convergence rate that is $O(\lambda^{-\frac{1}{2} +
  \epsilon})$, for all fixed $\epsilon > 0$, was proven for the maximum
difference between the two models over all $r \in \paren{\rb,R}$ and
$t \in \paren{\delta,\infty}$ (for any fixed $\delta > 0$).  Numerical
evaluation of the exact eigenfunction expansions, binding time
distributions, and mean binding times illustrated this convergence
rate, and demonstrated that for sufficiently large fixed values of
$\lambda$ the difference between the two models scaled like
$\rb^{-1}$ as $\rb$ was decreased. For biologically relevant values
of $\rb$, such as the reaction-radius for a protein diffusing to a
fixed DNA binding site, it was found that $\lambda$ should be chosen
at least as large as $10^{11} \textrm{s}^{-1}$ for the mean binding
time in the two models to differ by less than $1 \%$.

There are a number of extensions of the current work that would aid in
clarifying the rigorous relationship between the Smoluchowski and Doi
models. Foremost would be the study of more detailed, biologically
realistic, models in which multiple diffusing and reacting chemical
species are present. Such models require the introduction of unbinding
reactions, $\textrm{C} \to \textrm{A} + \textrm{B}$, which in the
Smoluchowski model require the use of unbinding radii (a separation,
greater than $\rb$, at which to place the newly created molecules to
avoid their immediate rebinding).  As numerical methods to solve the
Doi and Smoluchowski models have been used to study biological
systems, understanding how parameters in the Doi model should be
chosen so as to accurately approximate the Smoluchowski model would
greatly aid in comparing predictions by these different approaches.

\section{Acknowledgments} 
SAI and ICA are supported by NSF grant DMS-0920886.  ICA was also
supported by the Center for Biodynamics NSF RTG grant DMS-0602204.

\appendix

\section{Mean Binding Time} \label{S:appendix} Let $\avg{\Td}$ denote
the mean time at which the two molecules in the Doi
model~\eqref{eq:radpureDoi} first react when initially separated by
$r_0$.  $\avg{\Td}$ can be shown to
satisfy~\cite{GardinerHANDBOOKSTOCH,VanKampenSTOCHPROCESSINPHYS}
\begin{align} \label{eq:radpureDoiAvg}
  \Delta_{r_0} \avg{\Td}(r_0) - \hat{\lambda} \, \ind_{\{r < \rb\}}(r) \avg{\Td}(r_0) =-\frac{1}{D}, \quad 0 \leq r_0 < R, 
\end{align}
with the boundary condition,
\begin{align*}
  \PD{\avg{\Td}}{r_0}(R) = 0.
\end{align*}
(Here $\hat{\lambda} = \lambda / D$.) The solution to~\eqref{eq:radpureDoiAvg}
is given by~\eqref{eq:doiMeanBindTimeSolut}. 

The mean time, $\avg{\Ts}(r_0)$, at which the two molecules in the
Smoluchowski model~\eqref{eq:radpureSmol} first react when initially
separated by $r_0$ can be shown to
satisfy~\cite{GardinerHANDBOOKSTOCH,VanKampenSTOCHPROCESSINPHYS}
\begin{align} \label{eq:radpureSmolAvg}
 \Delta_{r_0} \avg{\Ts}(r_0) =-\frac{1}{D}, \quad \rb < r_0 < R,
\end{align}
with the boundary conditions,
\begin{align*}
  \avg{\Ts}(\rb) &= 0, && \PD{\avg{\Ts}}{r_0}(R) = 0.
\end{align*}
The solution to~\eqref{eq:radpureSmolAvg} is given
by~\eqref{eq:smolMeanBindTimeSolut}.

\section{Discrete Space and Time Points} \label{S:appendixB}
The spatial evaluation points, $r_i$, are generated in MATLAB by
\begin{center}
\begin{minipage}{0.98\textwidth}
\begin{lstlisting}[caption={$r_i$ points},label={lst:rpoints}]
    r1 = rb + rb*[5e-6 1e-5 5e-5 1e-4 5e-4 1e-3 5e-3]';
    r  = [r1; (.01:.01:1)' * (R-rb) + rb];
\end{lstlisting}
\end{minipage}
\end{center}

The time evaluation points, $t_j$, are generated in MATLAB by
\begin{center}
\begin{minipage}{.98\textwidth}
\begin{lstlisting}[caption={$t_j$ points},label={lst:tpoints}]
    t  = [1e-5 1e-4 1e-3 .01:.01:100 101:1:200 200:10:1000 1200:200:10000]';
\end{lstlisting}
\end{minipage}
\end{center}

\thispagestyle{empty} \bibliography{lib.bib}
\bibliographystyle{amsplain}

\end{document}